\documentclass[leqno,twoside]{article}
\usepackage[utf8]{inputenc}
\usepackage[english]{babel}
\usepackage[T1]{fontenc}

\usepackage{mathptmx}
\DeclareMathAlphabet{\mathscr}{OMS}{cmsy}{m}{n} 
\usepackage{geometry} 
\usepackage{enumitem} 
\setlist{nolistsep} 

\usepackage{booktabs} 
\usepackage{array} 
\usepackage[labelfont=bf]{caption}

\title{Formality and strongly unique enhancements}
\usepackage{fancyhdr} 
\makeatletter
\let\runtitle\@title
\makeatother
\usepackage{titletoc}
\titlecontents{section}[0em]{}{\thecontentslabel{.\hspace{6pt}}}{\hspace*{0em}}{\titlerule*[0.5pc]{.}\contentspage}

\usepackage[outermarks]{titlesec}

\titleformat{\section}
{\bfseries\Large}
{\thesection.}
{1ex}
{}
[]
\titleformat{\subsection}
{\normalfont
\titlerule*[.2em]{--}%
\vspace{4pt}%
\bfseries\footnotesize}
{}{0em}{\MakeUppercase}

\usepackage{tikz}
\usetikzlibrary{shapes,cd,decorations.pathmorphing}
\usepackage{amssymb}
\usepackage{amsmath}

\usepackage{calc}
\usepackage{mathtools}
\DeclarePairedDelimiter\abs{\lvert}{\rvert}
\DeclarePairedDelimiter\norm{\lVert}{\rVert}
\makeatletter
\let\oldabs\abs
\def\abs{\@ifstar{\oldabs}{\oldabs*}}
\let\oldnorm\norm
\def\norm{\@ifstar{\oldnorm}{\oldnorm*}}
\makeatother

\usepackage{amsthm}
\newtheoremstyle{plainmod}
{}
{}
{\it}
{}
{}
{}
{ }
{\textbf{\thmnumber{#2}.\thmname{\hspace{1ex}#1}}\thmnote{\textbf{ -- #3}}\textbf{.}}
\newtheoremstyle{definitionmod}
{}
{}
{}
{}
{}
{}
{ }
{\textbf{\thmnumber{#2}.\thmname{\hspace{1ex}#1}}\thmnote{\textbf{ -- #3}}\textbf{.}}
\newtheoremstyle{remarkmod}
{}
{}
{}
{}
{}
{}
{ }
{\textsc{\thmnumber{#2}.\thmname{\hspace{1ex}#1}}\thmnote{\textsc{ -- #3}}\textit{.}}

\numberwithin{equation}{section}

\theoremstyle{plainmod} 
\newtheorem{thm}[equation]{Theorem}
\newtheorem{cor}[equation]{Corollary}
\newtheorem{lem}[equation]{Lemma}
\newtheorem{prop}[equation]{Proposition}
\newtheorem{fact}[equation]{Fact}
\theoremstyle{definitionmod} 
\newtheorem{defn}[equation]{Definition}
\newtheorem{figr}[equation]{Figure}
\newtheorem{es}[equation]{Example}
\newtheorem{rem}[equation]{Remark}

\newtheorem{nota}[equation]{Notation}
\newtheorem{conv}[equation]{Convention}
\newtheorem{defnp}[equation]{Definition/Proposition}

\expandafter\let\expandafter\oldproof\csname\string\proof\endcsname
\let\oldendproof\endproof
\renewenvironment{proof}[1][\proofname]{%
  \oldproof[\scshape #1.]%
}{\oldendproof}

\usepackage{hyperref}
\usepackage[nottoc]{tocbibind}
\usepackage{bbm}
\usepackage{xcolor}
\usepackage{mathtools}
\usepackage{stackengine}
\setstackEOL{\\}

\newcommand{\coker}{\operatorname{coker}}

\newcommand{\acb}{\operatorname{Ac}^{\operatorname{b}}}
\newcommand{\derdg}{\operatorname{D}}
\newcommand{\der}{\mathscr{D}}
\newcommand{\derb}{\der^{\operatorname{b}}}
\newcommand{\kcom}{\mathscr{K}}
\newcommand{\kcomb}{\kcom^{\operatorname{b}}}
\newcommand{\dc}[1]{\derdg(#1)^{\operatorname{c}}}
\newcommand{\tr}{\operatorname{tr}}
\newcommand{\perf}{\operatorname{Perf}}
\newcommand{\id}{\operatorname{id}}

\newcommand{\essim}{\operatorname{EssIm}}
\newcommand{\dg}{\operatorname{DG}}
\newcommand{\ext}{\operatorname{Ext}}
\newcommand{\incl}{\operatorname{incl}}

\renewcommand{\hom}{\operatorname{Hom}}
\newcommand{\cone}{\operatorname{Cone}}
\newcommand{\ring}{\mathbbm{k}}
\newcommand{\pretr}{\operatorname{pretr}}

\newcommand{\fmod}{\operatorname{-mod}}
\newcommand{\Mod}{\operatorname{Mod}}
\newcommand{\dgmod}[1]{\operatorname{DGMod}(#1)}

\newcommand{\field}{\mathbb{K}}
\newcommand{\dgfun}{\mathcal{H}om}
\newcommand{\conedg}{\cone_{\dg}}
\newcommand{\Ind}{\operatorname{Ind}}
\newcommand{\semifree}{\operatorname{SF}}
\newcommand{\comdg}[1]{\mathsf{C}_{\dg}(#1)}
\newcommand{\comdgb}[1]{\mathsf{C}_{\dg}^{\operatorname{b}}(#1)}
\begin{document}
\titlepage
\begin{center}
{\LARGE \textbf{Formality and strongly unique enhancements}}
\par\vspace{1ex}
\textit{Antonio Lorenzin}
\par \vspace{1ex}
{\footnotesize Universit\"{a}t Innsbruck, Institut f\"{u}r Mathematik, Technikerstraße 25, 6020 Innsbruck, Austria
\par\vspace{-0.5ex}
E-mail address: \texttt{antonio.lorenzin@uibk.ac.at}}
\end{center}
\tableofcontents
\par\vspace{5ex}
\noindent {\large \textbf{Abstract}}\hspace{2ex} Inspired by the intrinsic formality of graded algebras, we prove a necessary and sufficient condition for strongly uniqueness of DG-enhancements. 
This approach offers a generalization to linearity over any commutative ring. In particular, we obtain several new examples of triangulated categories with a strongly unique DG-enhancement. Moreover, we also show that the bounded derived category of an exact category has a unique enhancement.
\par\vspace{2ex}

\noindent {\large \textbf{Keywords}}\hspace{2ex} triangulated categories; algebraic triangulated categories; strongly unique enhancements; DG-categories; standard categories
\par\vspace{2ex}

\noindent {\large \textbf{Acknowledgements}}\hspace{2ex} I am very thankful to my PhD supervisor, Alberto Canonaco, for his suggestions and corrections. I also wish to acknowledge the help of Francesco Genovese and Alice Rizzardo, whose suggestions clarified some arguments presented here. I would further like to thank Amnon Neeman for his insight into the proof of Proposition \ref{prop:realdg} and Corollary \ref{cor:dbunique}.
I am also grateful to Gustavo Jasso and Noah Olander for our brief discussions on their research.
Finally, I want to thank the referees, whose comments improved the exposition of the paper.

\section*{Introduction}\addcontentsline{toc}{section}{Introduction}
\markboth{Introduction}{Introduction}
It is well-known that cones of triangulated categories are not functorial. In order to fix the issues coming from such ill behaviour, the theory has developed the idea of enhancing triangulated categories to higher categories.

A triangulated category having a DG-enhancement, i.e. obtained by a DG-category with suspensions and {cones} (see Definition~\ref{def:suspension} and Definition~\ref{def:conedg}), is called \emph{algebraic}. As one may hope, many important examples of triangulated categories are algebraic: for instance, all derived categories and homotopy categories of complexes satisfy the definition.\footnote{As it will be made explicit in Convention \ref{conv:field}, all categories are assumed to be small in an appropriate Grothendieck universe.}

However, being algebraic does not prevent a triangulated category from having a bizarre behaviour. Indeed, a priori its structure may come from different DG-categories, which are not quasi-equivalent. This gives rise to the notion of having a unique DG-enhancement, satisfied by triangulated categories with one and only one DG-shape. Of course, not all algebraic triangulated categories have a unique enhancement. An example is given by $\mathbb{K}\fmod$, the category of finite dimensional vector spaces over the field $\mathbb{K}$. This category becomes triangulated with shift the identity and distinguished triangles generated by short exact sequences. In \cite{schlichting}, Schlichting proved that $\mathbb{K}\fmod$ does not have a unique enhancement when $\mathbb{K}=\mathbb{F}_p$ (with $p$ prime), giving two explicit examples that are not quasi-equivalent (for the transposition of the result in the DG-world, one may refer to \cite[Corollary 3.10]{canonacostellari}). It is important to notice that one of the DG-enhancements is not $\mathbb{K}$-linear. 

The motivating question that led to the birth of this article is the following: does $\mathbb{K}\fmod$ have a unique $\mathbb{K}$-linear enhancement? The answer is yes (Example \ref{es:Kmod}). In order to prove such result, we started by noticing that any DG-enhancement of $\mathbb{K}\fmod$ is associated to the graded algebra $\mathbb{K}[t,t^{-1}]$, where $t$ has degree 1. After some reading, we found the notion of intrinsic formality of a graded algebra (Definition \ref{def:intr}), which immediately implies uniqueness of the DG-enhancement for the associated triangulated category, as proved in Proposition \ref{prop:unique}. The conclusion that $\mathbb{K}[t,t^{-1}]$ is intrinsically formal follows from \cite[Corollary 4.8]{saito23}.

After this discovery, we wanted to understand how intrinsic formality relates to stricter requirements of the uniqueness of enhancements, namely strong uniqueness. This property tells us that the triangulated autoequivalences of the triangulated category come from the DG-world (cf. Proposition \ref{prop:strongenh}). Very few examples of triangulated categories with a strong unique enhancement are known, the most important one being the bounded derived categories of projective varieties, investigated by Lunts and Orlov in the celebrated article \cite{luntsorlov}. The procedure used to obtain such a result has been generalized for other cases: Canonaco and Stellari worked on coherent sheaves of a quasi-projective scheme supported in a projective subscheme \cite{canonacostellari}; Olander studied the case of smooth proper varieties over a field \cite{olanderorlov} (and later expanded his result in \cite{olanderphd}); Li, Pertusi, and Zhao considered the case of Kuznetsov components \cite{lipertusizhao}.

Currently, there is only one explicit example of a triangulated category with a unique but not strongly unique enhancement (see \cite[Corollary 5.4.12]{jassomuro}).
However, we do not know whether the uniqueness of enhancements implies the strong uniqueness of enhancements for derived categories or homotopy categories of complexes. It is worth noting that the examples of triangulated categories with a unique enhancement are by far more general: the most recent paper in this direction is \cite{canonaco2021uniqueness}, where it is proved that all the derived categories and all the homotopy categories of complexes over an abelian category have a unique enhancement.

For this reason, we are motivated to characterize strong uniqueness. With this aim, we define the original notions of triangulated formal DG-category, which is a relaxed version of intrinsic formality, and formally standard DG-category, inspired by D-standard and K-standard categories introduced by Chen and Ye in \cite{chenye}. When we restrict to graded categories, we have the following.

\par \vspace{1ex}
\noindent \textbf{\ref{thm:main}. Theorem.} \emph{Let $\mathsf{B}$ be a graded category. The following are equivalent:
                \begin{enumerate}
                    \item $\mathsf{B}$ is triangulated formal and formally standard;
                    \item $\tr(\mathsf{B})$ has a strongly unique enhancement;
                    \item $\dc{\mathsf{B}}$ has a strongly unique enhancement.
                \end{enumerate}}
\par \vspace{1ex}
In the statement, $\tr(\mathsf{B})$ indicates the triangulated category generated by $\mathsf{B}$ and $\dc{\mathsf{B}}$ is the idempotent completion of $\tr(\mathsf{B})$. In fact, the implication $1 \Rightarrow 2,3$ holds for more general DG-categories (see Proposition \ref{prop:sffssue} and Remark \ref{rem:perfsffssue}). It is very important to highlight that these results do not require linearity over a field (see Convention \ref{conv:field}).

First simple applications of the result are discussed in Section \ref{sec:freegen}. In Corollary \ref{cor:stronguniquefree} and Example \ref{es:Kmod}, we deal with particular cases of periodic triangulated categories, i.e. triangulated categories such that $[n]\cong \id$ for some integer $n$ (see \cite{saito23}).
In Section \ref{sec:kdstandard}, we show that K-standardness and D-standardness (Definition \ref{def:kdstandard}) are linked to formal standardness, and prove the following statements.

\par \vspace{1ex}
\noindent \textbf{\ref{prop:kfstandard}. Proposition.} \emph{An additive category $\mathcal{A}$ is K-standard if and only if $\kcomb(\mathcal{A})$ has a strongly unique enhancement.}
\par \vspace{1ex}

\noindent \textbf{\ref{thm:exdstand}. Theorem.} \emph{An exact category $\mathcal{E}$ is D-standard if and only if $\derb(\mathcal{E})$ has a strongly unique enhancement.}
\par \vspace{1ex}
These results follow from the fact that $\kcomb(\mathcal{A})$ and $\derb(\mathcal{E})$ have a (semi-strongly) unique enhancement for every choice of $\mathcal{A}$ additive and $\mathcal{E}$ exact (see Proposition \ref{prop:kcomba} and Proposition \ref{prop:derbssu}). We emphasize that the theorem above extends \cite[Theorem 5.10]{chenye}, which is valid for the category of finitely generated modules over a finite-dimensional $\field$-algebra, with $\field$ a field.

As an application, all bounded derived categories of hereditary abelian categories have a strongly unique enhancement (Corollary \ref{cor:hereditary}). For instance, $\derb(\Mod(\mathbb{Z}))$, the bounded derived category of abelian groups, has a strongly unique enhancement. Moreover, all geometric examples provided by the theory show D-standardness in order to conclude that the derived category has a strongly unique enhancement (see \cite{luntsorlov,canonacostellari,olanderorlov,olanderphd,lipertusizhao}). This follows by proving that the abelian category has an almost ample set, which is a generalization of the ample sequences introduced by Orlov in \cite{orlov1}.
\par \vspace{1ex}
\noindent \textbf{\ref{prop:almostample}. Proposition.} \emph{Let $\mathcal{A}$ be an abelian category with an almost ample set. Then $\derb(\mathcal{A})$ has a strongly unique enhancement.}
\par \vspace{2ex}

\noindent {\large \textbf{Overview}}\hspace{2ex} In Section \ref{sec:predg} and Section \ref{sec:dgenh}, we recall some basic notions and results of pretriangulated DG-categories and DG-enhancements. Section \ref{sec:sf} is devoted to the notion of triangulated formality of DG-categories, which generalizes intrinsic formality of graded algebras, while Section \ref{sec:fs} describes formal standardness. The main result, linking triangulated formality and formal standardness to strong uniqueness of DG-enhancements, is proved in Section \ref{sec:maintheorem}; applications are discussed in Section \ref{sec:freegen} and Section \ref{sec:kdstandard}. Appendix \ref{app:derbex} contains two results on bounded derived categories of exact categories.

\begin{conv}\label{conv:field}    
    All categories are assumed to be $\mathbb{U}$-small for an appropriate Grothendieck universe $\mathbb{U}$. 

    Moreover, we consider everything to be $\ring$-linear, where $\ring$ is a commutative ring. We reserve the capital letter $\field$ for fields.
    We adopt the term \emph{ring} when linearity is considered over a commutative ring, and \emph{algebra} for linearity over a field. In particular, we will talk about graded rings and graded algebras in the two contexts. 
\end{conv}
    \section{Pretriangulated DG-categories}\label{sec:predg}
    We start by recalling some basics on pretriangulated DG-categories. The non-expert reader may refer to \cite{chen2021informal,kelondg,toenondg,yekutieli} for detailed treatments on DG-categories.
    \begin{defn}
        A \emph{DG-category} is a $\ring$-linear category $\mathsf{C}$ such that $\hom(X,Y)$ is a $\mathbb{Z}$-graded $\ring$-module together with a differential $d:\hom(X,Y)\to \hom(X,Y)$ of degree 1 for any $X,Y \in \mathsf{C}$. The compositions $\hom(X,Y) \otimes \hom(Y,Z)\to \hom(X,Z)$ are morphisms of complexes ($X,Y,Z \in \mathsf{C}$).

        A morphism $f:X\to Y$ in a DG-category is \emph{closed} if $d(f)=0$, whereas it is \emph{homogeneous} if it belongs to $\hom(X,Y)^i \subset \hom(X,Y)$, the set of morphisms of degree $i$, for some $i\in \mathbb{Z}$. We will use $\abs{\cdot}$ to denote the degree of a homogeneous morphism.

        A \emph{DG-isomorphism} is a closed morphism of degree 0 with an inverse.
    \end{defn}
    \begin{defn}
        A functor $\mathsf{f}:\mathsf{C}\to \mathsf{D}$ between DG-categories is a \emph{DG-functor} if, for every $X,Y \in \mathsf{C}$,
    \[
    \mathsf{f}_{X,Y}  : \hom(X,Y) \to \hom (FX,FY) 
    \]
    is a morphism of degree 0 commuting with the differentials. We say that a DG-functor $\mathsf{f}$ is a \emph{DG-equivalence} if it is fully faithful and every object of $\mathsf{D}$ is DG-isomorphic to $\mathsf{f}(X)$ for some $X \in \mathsf{C}$.
    \end{defn}

    \begin{defn}\label{def:dgfun}
        Let $\mathsf{C},\mathsf{D}$ be two DG-categories. We define the \emph{functor DG-category}, denoted by $\dgfun(\mathsf{C},\mathsf{D})$, as follows: 
        \begin{itemize}
            \item The objects are the DG-functors $\mathsf{f}:\mathsf{C} \to \mathsf{D}$; 
            \item A morphism $\eta: \mathsf{f} \to \mathsf{g}$ of degree $i$ is a family of morphisms $(\eta_X: \mathsf{f}X \to \mathsf{g}X)_{X \in \mathsf{C}}$ of degree $i$ in $\mathsf{D}$ such that
        \begin{equation}\label{eq:dgnat}
        \mathsf{g}(h) \eta_X = (-1)^{\abs{\eta} \abs{h}} \eta_Y \mathsf{f}(h)  
        \end{equation}
        for any homogeneous $h:X \to Y$. With this definition, we can create $\hom(\mathsf{f}, \mathsf{g})$ as a graded module whose $i$-th component is the set of all morphisms of degree $i$. 
        The differential $(d\eta)_X := d(\eta_X)$
        gives $\dgfun(\mathsf{C}, \mathsf{D})$ the structure of a DG-category.
        \end{itemize}
        Morphisms of $\dgfun(\mathsf{C},\mathsf{D})$ are also called \emph{DG-natural transformations}. Similarly, its DG-isomorphisms are also called \emph{DG-natural isomorphisms}.
        \end{defn}
        \begin{rem}
            One may show that a DG-functor $\mathsf{f}:\mathsf{C}\to \mathsf{D}$ is a DG-equivalence if and only if there exists an inverse up to DG-natural isomorphisms. This result is analogous to the classical one for the equivalence of categories.
        \end{rem}
        \begin{defn}
            Let $\mathsf{C}$ be a DG-category. The \emph{opposite DG-category} $\mathsf{C}^o$ has the same objects of $\mathsf{C}$, morphisms have opposite direction $\hom_{\mathsf{C}^o}(X,Y) := \hom_{\mathsf{C}}(Y,X)$,
            and composition is reversed up to a change of sign:
            $
            f \circ_{\mathsf{C}^o} g := (-1)^{\abs{f}\abs{g}} g \circ_{\mathsf{C}} f$,
            where $f,g$ are homogeneous.
        \end{defn}
        \begin{defn}The \emph{homotopy category} $H^0(\mathsf{C})$ of a DG-category $\mathsf{C}$  has the same objects of $\mathsf{C}$ and 
            the morphisms are defined by $\hom_{H^0(\mathsf{C})}(X,Y):= H^0(\hom_{\mathsf{C}}(X,Y))$. Two objects $X,Y$ in $\mathsf{C}$ are \emph{homotopy equivalent} if they are isomorphic in $H^0(\mathsf{C})$.

The \emph{graded homotopy category} $H^*(\mathsf{C})$ has the same objects of $\mathsf{C}$ and $\hom_{H^*(\mathsf{C})} (X,Y):= H^*(\hom_{\mathsf{C}}(X,Y))$.

            Any DG-functor $\mathsf{f}:\mathsf{C}\to \mathsf{D}$ gives rise to functors $H^0(\mathsf{f}):H^0(\mathsf{C}) \to H^0 (\mathsf{D})$ and $H^*(\mathsf{f}):H^*(\mathsf{C})\to H^*(\mathsf{D})$.
            \end{defn}

            \begin{defn}\label{def:truncation}
                Let $\mathsf{C}$ be a DG-category. Its \emph{truncation} is the pair $(\tau_{\le 0}\mathsf{C},\mathsf{p}_{\le 0})$, where $\tau_{\le 0}\mathsf{C}$ is the DG-category whose objects are the same of $\mathsf{C}$ and
                \[
                 \hom_{\tau_{\le 0}\mathsf{C}} (X,Y)^n:= \begin{cases}
                     \hom(X,Y)^n & \text{if }n<0\\
                     \ker d^0_{\hom(X,Y)} & \text{if }n=0\\
                     0 & \text{if }n>0
                 \end{cases}   
                \]
                for every $X,Y \in \mathsf{C}$, while $\mathsf{p}_{\le 0}:\tau_{\le 0}\mathsf{C} \to \mathsf{C}$ is the natural DG-functor which is the identity on objects and the inclusion on morphisms.
             \end{defn}
             \begin{rem}\label{rem:lefttruncation}
                Dually, one may expect to define the truncation $(\tau_{\ge 0}\mathsf{C}, \mathsf{p}_{\ge 0}: \mathsf{C} \to \tau_{\ge 0} \mathsf{C})$ as follows: 
                \[ 
                \hom_{\tau_{\ge 0}\mathsf{C}}(X,Y)^n:= 
                \begin{cases}
                    0 & \text{if }n<0\\
                    \coker d^{-1}_{\hom(X,Y)} & \text{if }n=0\\
                    \hom(X,Y)^n & \text{if }n>0.
                \end{cases}  
                \] 
                However, we are unable to understand whether such a na\"{\i}ve construction makes sense in general, since the composition may not respect the cokernel.
                Moreover, even if this was the case, $\mathsf{p}_{\ge 0}$ is not always well-defined: pick two composable homogeneous morphisms $f$ and $g$, respectively of degree $-i$ and $n+i$ with $n,i>0$, such that $gf\ne 0$. Then $gf=\mathsf{p}_{\ge 0}(gf)\ne \mathsf{p}_{\ge 0}(g)\mathsf{p}_{\ge 0}(f)= \mathsf{p}_{\ge 0}(g) 0 =0$.\footnote{The author would like to thank Amnon Neeman and one of the referees for pointing out the impossibility of this dual truncation.}
        
                Despite such a situation, in the case of $\tau_{\le 0}\mathsf{C}$ the left truncation exists, because the DG-functor $\tau_{\le 0}\mathsf{C}\to H^0(\mathsf{C})$ is actually well-defined, since the above obstruction cannot happen.
             \end{rem}
             \begin{rem}\label{rem:truncation}
                Given a DG-functor $\mathsf{f}: \mathsf{C} \to \mathsf{D}$, we have a natural DG-functor $\tau_{\le 0}\mathsf{f}: \tau_{\le 0} \mathsf{C} \to \tau_{\le 0} \mathsf{D}$ satisfying $\mathsf{fp}_{\le 0} = \mathsf{p}_{\le 0}\tau_{\le 0}\mathsf{f}$. 
                \end{rem}

            \begin{defn}\label{es:cdg}
                Let $\mathcal{A}$ be a ($\ring$-linear) category. The \emph{DG-category of complexes (of $\mathcal{A}$)}, denoted with $\comdg{\mathcal{A}}$, is described as follows:
                \begin{itemize}
                    \item Its objects are complexes $M=(M^i,d_M^i)_{i \in \mathbb{Z}}$, i.e. sequences 
                    \[
                        \begin{tikzcd}
                            \dots \ar[r]& M^{i-1}\ar[r,"d_M^{i-1}"] & M^i \ar[r,"d_M^i"] & M^{i+1}\ar[r]&\dots 
                        \end{tikzcd}
                    \]
                    with $\lbrace M^i \rbrace_{i\in \mathbb{Z}}\subset \mathcal{A}$ and $d_M^{i}d_M^{i-1}=0$ for all $i$.
                    \item Given $M,N$ complexes, the $\ell$-th homogeneous morphisms are \[\hom_{\comdg{\mathcal{A}}}(M,N)^\ell:=\prod_{i \in \mathbb{Z}} \hom_{\mathcal{A}}(M^i , N^{\ell+i}).\]
                    The differential is defined by
                    \[
               d(f):= d_N \circ f - (-1)^{\abs{f}} f d_M
                        \]
                        for any homogeneous element $f$.
                \end{itemize}
                Analogously, we can define the \emph{DG-category of bounded complexes} $\comdgb{\mathcal{A}}$ as the full DG-subcategory of $\comdg{\mathcal{A}}$ whose objects are bounded complexes, i.e. $M=(M^i,d^i)$ with $M^i=0$ for $\abs{i} \gg 0$. 

                An important example is given by $\operatorname{Mod}(\ring)$, the category of all $\ring$-modules; $\comdg{\operatorname{Mod}(\ring)}$ is called \emph{the category of DG-modules}.
            \end{defn}
        
        \begin{defn}
        Let $\mathsf{C}$ be a DG-category. The \emph{DG-category of (right) DG $\mathsf{C}$-modules} is defined by $\dgmod{\mathsf{C}}:=\dgfun(\mathsf{C}^o,\comdg{\operatorname{Mod}(\ring)})$.
        \end{defn}
        \begin{lem}[Yoneda DG-embedding]\label{lem:yoneda}
            The Yoneda embedding
            \[
            \mathsf{y} : \mathsf{C} \to \dgmod{\mathsf{C}}  : X \mapsto \hom_{\mathsf{C}}(-,X)   
            \]    
            is a fully faithful DG-functor.
        \end{lem}

    \begin{defn}\label{def:suspension}
        Let $\mathsf{C}$ be a DG-category.
        Given $X \in \mathsf{C}$, its \emph{suspension} is an object $Y \in \mathsf{C}$ equipped with two closed morphisms $\sigma_X :Y \to X$ of degree 1 and $\tau_X : X\to Y$ of degree $-1$ such that $\tau_X\sigma_X=\id_Y$ and $\sigma_X \tau_X=\id_X$. In this case, we will write $\Sigma(X):=Y$, since $Y$ is determined up to unique DG-isomorphism.
    
    \end{defn}
    \begin{defn}\label{def:conedg}
        Let $f:X\to Y$ be a closed morphism of degree 0 in a DG-category $\mathsf{C}$. The \emph{cone} of $f$ is an object $\conedg(f)$ together with 4 morphisms of degree 0
        \[
        \begin{tikzcd}[column sep=large]
            Y \ar[r,"j"] & \conedg(f)\ar[l,bend left, dashed, "q"] \ar[r,"p"] & \Sigma(X)\ar[l,bend left,dashed,"i"]
        \end{tikzcd}
        \]
    satisfying 
    \[
        pi=\id,\quad qj=\id,\quad qi=0,\quad pj=0,\quad ip+jq=\id
    \]
    and 
    \[
    d(j)=d(p)=0,\quad d(i)=jf\sigma_X,\quad d(q)=-f\sigma_X p.    
    \]
    We emphasize that $j$ and $p$ are required to be closed, while $q$ and $i$ are generally not.
    \end{defn}
    
    \begin{defn}\label{def:strong_pretriangulated}
        A DG-category $\mathsf{C}$ is \emph{strongly pretriangulated} if it is closed under cones and suspensions, meaning that:
        \begin{itemize}
            \item Any object $X\in \mathsf{C}$ has a suspension $Y\in \mathsf{C}$, and there exists $Z\in \mathsf{C}$ such that $X$ is a suspension for $Z$;
            \item Any closed morphism of degree 0 has a cone.
        \end{itemize}
    \end{defn}

    The previous definitions are inspired by the structure of DG-modules. In fact, one can prove the following well-known result.
    

    \begin{prop}\label{prop:dgmodtri}
        For any DG-category $\mathsf{C}$,  $\dgmod{\mathsf{C}}$ is strongly pretriangulated.
    \end{prop}
    \begin{defn}\label{def:pretrclosure}
        For any DG-category $\mathsf{C}$, its \emph{pretriangulated closure} is the smallest full DG-subcategory $\mathsf{C}^{\pretr}$ of $\dgmod{\mathsf{C}}$ that contains $\mathsf{y}(\mathsf{C})$ and is strongly pretriangulated. In particular, notice that the Yoneda embedding factors through $\mathsf{C}^{\pretr}$. By an abuse of notation, we call $\mathsf{y} :\mathsf{C}\hookrightarrow \mathsf{C}^{\pretr}$ the natural embedding so obtained. One may prove that a DG-category $\mathsf{C}$ is strongly pretriangulated if and only if $\mathsf{y}:\mathsf{C} \to \mathsf{C}^{\pretr}$ is a DG-equivalence (see \cite[Lemma 6.2.2]{chen2021informal}).
    \end{defn}
    \begin{rem}\label{rem:extfunpretr}
        Given a DG-functor $\mathsf{f}: \mathsf{C} \to \mathsf{D}$, we have an induced DG-functor $\mathsf{f}^{\pretr} : \mathsf{C}^{\pretr} \to \mathsf{D}^{\pretr}$ given by restricting $\Ind(\mathsf{f}):\dgmod{\mathsf{C}} \to \dgmod{\mathsf{D}}$, defined in \cite[C.9]{drinfeld}.
    \end{rem}

    \begin{es}\label{es:comdgpretr}
        A direct check shows that $\comdg{\mathcal{A}}$ is a strongly pretriangulated DG-category for any additive category $\mathcal{A}$. The same holds for $\comdgb{\mathcal{A}}$.
    \end{es}
    An important tool for our studies is the following universal property of the pretriangulated closure.
    \begin{prop}\label{prop:dgextension}\emph{\cite[Section 4.5]{kelondg}}
        Let $\mathsf{C}, \mathsf{D}$ be DG-categories, with $\mathsf{D}$ strongly pretriangulated. Then the embedding $\mathsf{y}:\mathsf{C} \hookrightarrow \mathsf{C}^{\pretr}$ induces a DG-equivalence $\dgfun(\mathsf{C}^{\pretr}, \mathsf{D}) \to \dgfun(\mathsf{C},\mathsf{D})$.
        
        More explicitly, 
        
        \noindent\emph{(Essential surjectivity).} Any DG-functor $\mathsf{f}: \mathsf{C} \to \mathsf{D}$ admits an extension $\mathsf{f}' : \mathsf{C}^{\pretr} \to \mathsf{D}$.

        \noindent\emph{(Fully faithfulness).} Let us consider another DG-functor $\mathsf{g}: \mathsf{C} \to \mathsf{D}$ and an extension $\mathsf{g}': \mathsf{C}^{\pretr}\to \mathsf{D}$. Then any DG-natural transformation $\eta: \mathsf{f} \to \mathsf{g}$ extends to a unique DG-natural transformation $\eta':\mathsf{f}' \to \mathsf{g}'$. In particular, if $\eta$ is a DG-natural isomorphism, then $\eta'$ is a DG-natural isomorphism. 
    \end{prop}

    \begin{defn}
        A DG-functor $\mathsf{f}:\mathsf{C}\to \mathsf{D}$ is 
        \begin{itemize}
            \item \emph{quasi-fully faithful} if the induced graded functor $H^*(\mathsf{f}):H^*(\mathsf{C}) \to H^*(\mathsf{D})$ is fully faithful;
            \item \emph{quasi-essentially surjective} if $H^0(\mathsf{f}): H^0(\mathsf{C}) \to H^0(\mathsf{D})$ is essentially surjective;
            \item a \emph{quasi-equivalence} if it is quasi-fully faithful and quasi-essentially surjective.
        \end{itemize}
        From now on, $\simeq$ is used to indicate that a DG-functor is a quasi-equivalence.
    \end{defn}
    \begin{rem}\label{rem:h0qespretr}
        If a DG-functor $\mathsf{f}:\mathsf{C}\to \mathsf{D}$ is a quasi-equivalence, then also $\mathsf{f}^{\pretr}$ is a quasi-equivalence (see \cite[Proposition 2.5]{drinfeld}). Similarly, if $\mathsf{f}$ is fully faithful, then $\mathsf{f}^{\pretr}$ is fully faithful as well.

        Moreover, using suspensions, one can show that a DG-functor $\mathsf{f}$ between strongly pretriangulated DG-categories is a quasi-equivalence if and only if $H^0(\mathsf{f})$ is an equivalence. 
    \end{rem}
    
    \begin{defn}\label{def:qfunctor}
        A \emph{quasi-functor} $\mathsf{f}:\mathsf{C}\to \mathsf{D}$ is a (suitable equivalence class of) zig-zag of DG-functors
        \[
        \begin{tikzcd}
            & \mathsf{C}_1 \arrow{ld}[sloped,above]{\simeq}\ar[rd] && \dots \arrow{ld}[sloped,above]{\simeq}\ar[rd]& \\
            \mathsf{C}&&\mathsf{C}_2&&\mathsf{D}.
        \end{tikzcd}
        \]
        A \emph{quasi-equivalence quasi-functor} is a quasi-functor given by a zig-zag of quasi-equivalences. Two DG-categories are \emph{quasi-equivalent} if there exists a quasi-equivalence quasi-functor.
        In particular, a quasi-functor is a morphism of the homotopy category with respect to the model structure on the category of DG-categories with weak equivalences the quasi-equivalences (cf. \cite[Definition 1.2.1]{hovey} and 
        \cite{tabuada}).\footnote{In the literature, quasi-functors are generally a zig-zag of DG-functors as above, and our definition of quasi-functor is often called an isomorphism class of quasi-functors.}
    \end{defn}
    \begin{rem} 
        According to Remark \ref{rem:extfunpretr} and Remark \ref{rem:h0qespretr}, for every quasi-functor $\mathsf{f}: \mathsf{C}\to \mathsf{D}$ we have an induced quasi-functor $\mathsf{f}^{\pretr} : \mathsf{C}^{\pretr} \to \mathsf{D}^{\pretr}$. 
    \end{rem}
    \begin{rem}\label{rem:0qf}
        The only quasi-functor $\mathsf{f}:\mathsf{A}\to \mathsf{B}$ such that $H^0(\mathsf{f})=0$ is the trivial one.

        Indeed, if $H^0(\mathsf{f})=0$, then the image of $\mathsf{f}$ is quasi-equivalent to $0$. We conclude that $\mathsf{f} = 0$ by the following commutative diagram:
        \[
        \begin{tikzcd}
            \mathsf{A} \ar[rr,bend left=40, "\mathsf{f}"]\ar[r]\ar[dr] & \operatorname{im} (\mathsf{f}) \ar[r,hook]\ar[d,"\simeq"] & \mathsf{B}.\\
            & 0 &
        \end{tikzcd}    
        \]
    \end{rem}
    \begin{defn}
    A DG-category $\mathsf{C}$ is \emph{pretriangulated} if $H^0(\mathsf{y}):H^0(\mathsf{C}) \to H^0(\mathsf{C}^{\pretr})$ is an equivalence. 
    \end{defn}
    \begin{prop}\emph{\cite[Proposition 1]{bondalkapranov}.}
        Let $\mathsf{C}$ be a pretriangulated DG-category. Then 
        $H^0(\mathsf{C})$ is a triangulated category. 
    \end{prop}
    \begin{rem}
    The second part of Remark \ref{rem:h0qespretr} can be generalized to quasi-functors between pretriangulated DG-categories using $H^0(\mathsf{y})$.
    \end{rem}
    \begin{defn}
        Let $\mathsf{C}$ be a DG-category. The \emph{DG-category of semi-free DG-modules} $\semifree (\mathsf{C})$ is the full DG-subcategory of $\dgmod{\mathsf{C}}$ whose objects have a filtration of free DG $\mathsf{C}$-modules, i.e.
        \[
        0 = M_0 \subset M_1 \subset \dots \subset M_{i-1}\subset M_i \subset \dots = M     
        \]
        such that $M_i / M_{i-1}$ is isomorphic to a direct sum of $\Sigma^n\mathsf{y}(X)$ for some $n \in \mathbb{Z}$ and $X \in \mathsf{C}$.
    \end{defn}
    
    \begin{defn}
        Let $\mathsf{C}$ be a DG-category. The \emph{DG-category of perfect complexes} $\perf(\mathsf{C})$ is the full DG-subcategory of $\semifree(\mathsf{C})$ whose objects are homotopy equivalent to a direct summand of $\mathsf{C} ^{\pretr}$. 
    \end{defn}
    \begin{rem}
        In the same fashion of Remark \ref{rem:extfunpretr}, 
        any DG-functor $\mathsf{f}: \mathsf{C} \to \mathsf{D}$ admits extensions $\perf(\mathsf{f}) : \perf(\mathsf{C})\to \perf(\mathsf{D})$ and $\semifree(\mathsf{f}): \semifree(\mathsf{C}) \to \semifree (\mathsf{D})$, given by restricting the DG-functor $\Ind(\mathsf{f}):\dgmod{\mathsf{C}} \to \dgmod{\mathsf{D}}$, defined in \cite[C.9]{drinfeld}.
    \end{rem}
    \begin{rem}
        The following chain of inclusions of strongly pretriangulated DG-categories holds:
        \[
        \mathsf{C}^{\pretr} \subset \perf (\mathsf{C}) \subset \semifree(\mathsf{C}) \subset \dgmod{\mathsf{C}}.    
        \]
    \end{rem}
     \begin{nota}\label{nota:basictri}
    Let $\mathsf{C}$ be a DG-category. We will use the following notation: 
    \[\tr(\mathsf{C}):= H^0(\mathsf{C}^{\pretr}), \quad \dc{\mathsf{C}}:=H^0 (\perf (\mathsf{C})), \quad \derdg(\mathsf{C}):= H^0 (\semifree(\mathsf{C})).
    \]
    \end{nota}
    \begin{rem}\label{rem:dc}
        The notation $\derdg (\mathsf{C})$ for $H^0(\semifree ({\mathsf{C}}))$ is motivated by \cite[Theorem 3.1 and Section 4.1]{kellerderiving} (see also \cite[Appendices B and C]{drinfeld}), where it is shown that this triangulated category is in fact equivalent to the Verdier quotient of $H^0(\dgmod{\mathsf{C}})$ by the subcategory of acyclic DG-modules.

        Also, the notation $\dc{\mathsf{C}}$ comes from the fact that $H^0(\perf(\mathsf{C}))$ is given by all the \emph{compact objects} of $\derdg(\mathsf{C})$ (\cite{k1}). We recall that an object $X$ in a triangulated category with all coproducts is compact if $\hom(X,-)$ commutes with coproducts. 
        
        Another important fact to keep in mind is that $\dc{\mathsf{C}}$ is the idempotent closure of $\tr(\mathsf{C})$.
    \end{rem}

    \begin{defnp}\cite[Theorem 1.6.2]{drinfeld}.\label{defp:dgquotient}
        Let $\mathsf{C}$ be a DG-category and $\mathsf{B}\subset \mathsf{C}$ a full DG-subcategory. A \emph{DG-quotient},  often denoted by $\mathsf{C}/\mathsf{B}$, is a DG-category $\mathsf{D}$ together with a quasi-functor $\mathsf{q}:\mathsf{C}\to \mathsf{D}$ satisfying the following equivalent properties:
        \begin{enumerate}
          \item The functor $H^0(\mathsf{q})$ is essentially surjective and $H^0(\mathsf{q}^{\pretr})$ induces a triangulated equivalence $\tr(\mathsf{C})/\tr(\mathsf{B}) \to \tr(\mathsf{D})$.
          \item For every DG-category $\mathsf{K}$, the category of quasi-functors $\mathsf{D}\to \mathsf{K}$ is equivalent by composition to the category of quasi-functors $\mathsf{C}\to \mathsf{K}$ such that $\mathsf{B}\to \mathsf{C}\to \mathsf{K}$ is zero (see \cite[Appendix E]{drinfeld} for a discussion on these categories).
        \end{enumerate}
        The DG-quotient is determined up to quasi-equivalence, i.e. given another DG-quotient $\mathsf{D}'$ with $\mathsf{q}': \mathsf{C}\to \mathsf{D}'$, we can find a quasi-equivalence quasi-functor $\mathsf{f}: \mathsf{D}\to \mathsf{D}'$ such that $\mathsf{q}' \cong \mathsf{fq}$.
      \end{defnp}
      \begin{rem}
      With the same notation above, we can choose $\mathsf{D}$ so that $\mathsf{q}$ becomes a DG-functor. Additionally, if $\mathsf{C}$ is pretriangulated, then so is $\mathsf{D}$. The reader may refer to \cite[Remark 1.4 and Lemma 1.5]{luntsorlov}. 
      \end{rem}

    \begin{es}[Main examples]\label{es:homder}
        \item 
        \begin{itemize}
            \item Let $\mathcal{A}$ be any additive category. Let us consider $\comdg{\mathcal{A}}$, whose homotopy category is the homotopy category of complexes $\kcom(\mathcal{A})$. By Proposition \ref{prop:dgextension} and Remark \ref{rem:h0qespretr}, the inclusion $\mathcal{A} \to \comdg{\mathcal{A}}$ extends to a fully faithful DG-functor $\mathcal{A}^{\pretr} \to \comdg{\mathcal{A}}$ because $\comdg{\mathcal{A}}$ is strongly pretriangulated (as noted in Example \ref{es:comdgpretr}). Therefore, $\tr(\mathcal{A})$ is equivalent to a triangulated subcategory of $\kcom(\mathcal{A})$. Moreover, it is the triangulated envelope of $\mathcal{A}$.\footnote{The notion of triangulated envelope is defined in \cite[Section 2]{bondlarslunts}.} This suffices to conclude that $\tr(\mathcal{A})\cong \kcomb(\mathcal{A})$, since $\kcomb(\mathcal{A})$ is the triangulated envelope of $\mathcal{A}$ in $\kcom(\mathcal{A})$. In addition, $\comdgb{\mathcal{A}}$ is quasi-equivalent to $\mathcal{A}^{\pretr}$.
            
            For a general ($\ring$-linear) category $\mathcal{A}$, the same argument proves that $\tr(\mathcal{A})$ is equivalent to the bounded homotopy category of complexes in its additive closure.
            \item Let $\mathcal{E}$ be an exact category.  
              Let $\acb_{\dg}{(\mathcal{E})}$ be the full DG-subcategory of $\comdgb{\mathcal{E}}$ whose objects are acyclic complexes. 
            Consider the DG-quotient $\derb_{\dg} (\mathcal{E}):=\comdgb{\mathcal{E}} / \acb_{\dg}{(\mathcal{E})}$.
            Its homotopy category is equivalent to $\derb(\mathcal{E})$ since 
              \[H^0(\derb_{\dg} (\mathcal{E}))\cong H^0(\comdgb{\mathcal{E}}) / H^0(\acb_{\dg}{(\mathcal{E})})= \kcomb(\mathcal{E})/\acb(\mathcal{E}),\]
            and $\derb(\mathcal{E})$ is by definition the Verdier quotient $\kcomb(\mathcal{E})/\acb(\mathcal{E})$ (see, for instance, \cite[Chapter 10]{buhler} for an introduction to derived categories of exact categories).
        \end{itemize}
    \end{es}

    \section{DG-enhancements}\label{sec:dgenh}
    In this section, we state some properties and results about DG-enhancements. In particular, we introduce the notion of lift to discuss the relation between strong uniqueness of enhancements and autoequivalences.
    \begin{defn}
        Let $\mathcal{T}$ be a triangulated category. A \emph{(DG-)enhancement} is a pair $(\mathsf{C},E)$ where $\mathsf{C}$ is a pretriangulated DG-category and $E:H^0(\mathsf{C}) \to \mathcal{T}$ is a triangulated equivalence. If a triangulated category admits a (DG-)enhancement, it is called \emph{algebraic}.
    \end{defn}
    \begin{defn}
        An algebraic triangulated category $\mathcal{T}$ has a \emph{unique (DG-)enhancement} if, given two enhancements $(\mathsf{C},E)$ and $(\mathsf{C}',E')$, there exists a quasi-equivalence quasi-functor $\mathsf{f}:\mathsf{C} \to \mathsf{C}'$. In other words, there exists a zig-zag of quasi-equivalences
        \[
        \begin{tikzcd}
            & \mathsf{D}_1 \arrow{ld}[sloped,above]{\simeq}\ar[rd,"\simeq" sloped] && \dots \arrow{ld}[sloped,above]{\simeq}\ar[rd,"\simeq" sloped]& \\
            \mathsf{C}&&\mathsf{D}_2&&\mathsf{C}'.
        \end{tikzcd}
        \]
    \end{defn}
    \begin{defn}
        An algebraic triangulated category $\mathcal{T}$ has a \emph{strongly unique enhancement} (respectively, \emph{semi-strongly unique}) if, given two enhancements $(\mathsf{C},E)$ and $(\mathsf{C}',E')$, there exists a quasi-equivalence quasi-functor $\mathsf{f}:\mathsf{C} \to \mathsf{C}'$ such that $E\cong E' H^0(\mathsf{f})$ (respectively, $E(X)\cong E' H^0(\mathsf{f})(X)$ for all $X \in \mathsf{C}$).
    \end{defn}
    
    We want to interpret strong and semi-strong uniqueness of enhancements via autoequivalences of the triangulated category. Let us start with an unconventional definition.
    
    \begin{defn}
        Let $F: \mathcal{T}\to \mathcal{T}'$ be a triangulated functor between algebraic triangulated categories. Given enhancements $(\mathsf{C},E)$ and $(\mathsf{C}',E')$ of $\mathcal{T}$ and $\mathcal{T}'$ respectively, a \emph{$(\mathsf{C},E)-(\mathsf{C}',E')$-lift} of $F$ is a quasi-functor $\mathsf{f}:\mathsf{C}\to \mathsf{C}'$ such that $F\cong E' H^0(\mathsf{f}) E^{-1}$. 
        We say that $\mathsf{f}$ is a \emph{$(\mathsf{C},E)$-lift} of $F$ when $\mathcal{T}=\mathcal{T}'$ and $(\mathsf{C}',E')=(\mathsf{C},E)$; and $F$ has a \emph{good DG-lift} if for every enhancement $(\mathsf{C},E)$, there exists a $(\mathsf{C},E)$-lift for $F$. 
        
        Similarly, a \emph{$(\mathsf{C},E)-(\mathsf{C}',E')$-semilift} of $F$ is a quasi-functor  $\mathsf{f}:\mathsf{C}\to \mathsf{C}'$ such that $F(X)\cong E' H^0(\mathsf{f}) E^{-1}(X)$ for any $X \in \mathcal{T}$. In the same fashion as above, we define \emph{$(\mathsf{C},E)$-semilift} and \emph{good DG-semilift}.    
    \end{defn}
    
    \begin{prop}\label{prop:strongenh}
    Let $\mathcal{T}$ be an algebraic triangulated category with a unique enhancement. The following are equivalent.
    \begin{enumerate}
    \item $\mathcal{T}$ has a strongly unique enhancement.
    \item Every triangulated autoequivalence $F:\mathcal{T} \to \mathcal{T}$ has a good DG-lift.
    \item There exists an enhancement $(\mathsf{C},E)$ such that any triangulated autoequivalence $F:\mathcal{T} \to \mathcal{T}$ has a $(\mathsf{C},E)$-lift.
    \item There exist two enhancements $(\mathsf{C},E)$ and $(\mathsf{C}',E')$ such that any triangulated autoequivalence $F:\mathcal{T} \to \mathcal{T}$ has a $(\mathsf{C},E)-(\mathsf{C}',E')$-lift.
    \end{enumerate}
    \end{prop}
    This result should be well-known to experts, although we cannot find a proper reference.
    \begin{proof}
    1 $\Rightarrow$ 4. It suffices to take some arbitrary enhancements $(\mathsf{C},FE)$ and $(\mathsf{C}',E')$. Indeed, by hypothesis, we can find $\mathsf{f}$ such that $FE\cong E' H^0(\mathsf{f})$, so $F \cong E' H^0(\mathsf{f})E^{-1}$.

    4 $\Rightarrow$ 3. Let us consider a quasi-equivalence quasi-functor $\mathsf{g}: \mathsf{C}' \to \mathsf{C}$. This gives an equivalence $G=EH^0(\mathsf{g})(E')^{-1}:\mathcal{T}\to \mathcal{T}$. We now consider $\mathsf{h}$ a $(\mathsf{C},E)-(\mathsf{C}',E')$-lift for $G^{-1}F$. The quasi-functor $\mathsf{gh}:\mathsf{C}\to \mathsf{C}$ is then a $(\mathsf{C},E)$-lift for $F$:
    \begin{equation*}
    \begin{split}
    EH^{0}(\mathsf{gh}) E^{-1} &\cong EH^0(\mathsf{g})H^0(\mathsf{h})E^{-1} \\
    & \cong EH^0(\mathsf{g})(E')^{-1} E' H^0(\mathsf{h})E^{-1}\\
    & \cong G G^{-1} F\cong F
    \end{split}
    \end{equation*}

    3 $\Rightarrow$ 2. Let $(\mathsf{C}',E')$ be another enhancement and let $\mathsf{g}:\mathsf{C}'\to \mathsf{C}$ be a quasi-equivalence quasi-functor. We denote with $L$ the autoequivalence $\mathcal{T} \to \mathcal{T}$ given by $EH^0(\mathsf{g})(E')^{-1}$. From assumption, $LFL^{-1}$ has a $(\mathsf{C},E)$-lift $\mathsf{h}$. The quasi-functor $\mathsf{g}^{-1}\mathsf{hg}: \mathsf{C}' \to \mathsf{C}'$ is the wanted $(\mathsf{C}',E')$-lift: 
    \begin{equation*}
    \begin{split}
    E'H^0(\mathsf{g}^{-1}\mathsf{hg})(E')^{-1} &\cong E' H^0(\mathsf{g}^{-1})H^0(\mathsf{h})H^0(\mathsf{g})(E')^{-1} \\
    &\cong L^{-1} EH^0(\mathsf{h}) E^{-1} L \\
    &\cong L^{-1} LFL^{-1} L \cong F.
    \end{split}
    \end{equation*}

    2 $\Rightarrow$ 1. Let $(\mathsf{C},E)$ and $(\mathsf{C}', E')$ be two enhancements. By assumption, we have a quasi-equivalence quasi-functor $\mathsf{g}:\mathsf{C} \to \mathsf{C}'$. Now we consider the equivalence $E H^0(\mathsf{g})^{-1} (E')^{-1}: \mathcal{T} \to \mathcal{T}$. Since it has a good DG-lift, we can consider its $(\mathsf{C},E)$-lift $\mathsf{f}$, satisfying $EH^0(\mathsf{f})E^{-1} \cong E H^0(\mathsf{g})^{-1} (E')^{-1}$. From this isomorphism of functors, we obtain $E' H^0(\mathsf{gf}) \cong E$.
    \end{proof}
    
    Analogously to the previous proposition, one can show the following.
    \begin{prop}\label{prop:semistrongenh}
        Let $\mathcal{T}$ be an algebraic triangulated category with a unique enhancement. The following are equivalent.
        \begin{enumerate}
        \item $\mathcal{T}$ has a semi-strongly unique enhancement.
        \item Every triangulated autoequivalence $F:\mathcal{T} \to \mathcal{T}$ has a good DG-semilift.
        \item There exists an enhancement $(\mathsf{C},E)$ such that any triangulated autoequivalence $F:\mathcal{T} \to \mathcal{T}$ has a $(\mathsf{C},E)$-semilift.
        \item There exist two enhancements $(\mathsf{C},E)$ and $(\mathsf{C}',E')$ such that any triangulated autoequivalence $F:\mathcal{T} \to \mathcal{T}$ has a $(\mathsf{C},E)-(\mathsf{C}',E')$-semilift.
        \end{enumerate}
        \end{prop}

            \begin{nota}\label{nota:restrictdg}
                Let $(\mathsf{C},E)$ be an enhancement of a triangulated category $\mathcal{T}$, and let $\mathcal{S} \subset \mathcal{T}$ be a full subcategory (not necessarily triangulated). With the notation $\mathsf{C}_{\mid \mathcal{S}}^E$ we mean the full DG-subcategory of $\mathsf{C}$ whose objects $X$ are such that $E(X)\cong Y \in \mathcal{S}$. In particular, $E H^0(\mathsf{C}^E_{\mid \mathcal{S}})$ is equivalent to $\mathcal{S}$. When there is no misunderstanding, we simply write $\mathsf{C}_{\mid \mathcal{S}}$ instead of $\mathsf{C}_{\mid \mathcal{S}}^E$.
            \end{nota}
            \begin{lem}\label{lem:restrictdg}
                Let $(\mathsf{C},E)$ be an enhancement of a triangulated category $\mathcal{T}$ and consider $\mathcal{S} \subset \mathcal{T}$ a full subcategory. Then $\mathsf{C}_{\mid \mathcal{S}}$ is closed under homotopy equivalence in $\mathsf{C}$.
            \end{lem}
            \begin{proof}
                Let $Y \in \mathsf{C}$ be homotopy equivalent to $X \in \mathsf{C}_{\mid \mathcal{S}}$. In other words, $X$ and $Y$ are isomorphic in $H^0(\mathsf{C})$. Since an equivalence sends isomorphisms to isomorphisms, $E(Y)\cong E(X)$. By definition, there exists $Z \in \mathcal{S}$ such that $E(X) \cong Z$. We obtain $E(Y)\cong Z$, which implies that $Y \in \mathsf{C}_{\mid \mathcal{S}}$.
            \end{proof}

            We now investigate some relations between uniqueness of enhancements of the triangulated categories associated to a DG-category. We start with a technical lemma.
            \begin{lem}\label{lem:dgtaylor}
                Let $\mathcal{T}$ be a triangulated category and $\mathsf{A}$ be a DG-category. If \[F:\mathcal{T} \to H^0(\dgmod{\mathsf{A}})\] is a full triangulated functor and $H^0(\mathsf{A})\subset \essim(F)$, then $\tr(\mathsf{A}) \subset \essim(F)$.
            
                Moreover, if $\mathcal{T}$ is idempotent complete and $G:\mathcal{T} \to H^0(\dgmod{\mathsf{A}})$ is a fully faithful triangulated functor such that $H^0(\mathsf{A})\subset \essim(G)$, then  $\dc{\mathsf{A}} \subset \essim(G)$.
            \end{lem}
            \begin{proof}
                Let $X \in \tr(\mathsf{A})$. We aim to show that $X \in \essim (F)$. If $X \in H^0(\mathsf{A})$, this is true by hypothesis. Therefore, it suffices to show that $\essim(F)$ is closed under shifts and cones (cf. Definition \ref{def:pretrclosure}). 
                A trivial reasoning shows that $X\in \essim(F)$ if it is the shift of an object in $\essim(F)$.
                It remains to study the case $X= \cone (f)$ for a morphism $f:Y_1 \to Y_2$ where $Y_i \in \essim(F)$ for $i=1,2$. Notice there exist an object $Z_i \in \mathcal{T}$ and an isomorphism $\varphi_i:FZ_i \to Y_i$. Since $F$ is full, we can find $g:Z_1 \to Z_2$ such that $Fg= \varphi_2^{-1} f \varphi_1$. Then $X \cong \cone(Fg)\cong F(\cone(g)) \in \essim(F)$. 
            
                When $\mathcal{T}$ is idempotent complete and $G:\mathcal{T} \to H^0(\dgmod{\mathsf{A}})$ is a fully faithful triangulated functor, $\tr(\mathsf{A}) \subset \essim(G)$ directly implies that $\dc{\mathsf{A}} \subset \essim(G)$.
            \end{proof}
            \begin{rem}\label{rem:dgtaylor}
                It directly follows from Lemma \ref{lem:dgtaylor} that every fully faithful triangulated functor $F:\tr(\mathsf{A}) \to \tr(\mathsf{A})$ with $H^0(\mathsf{A}) \subset \essim(F)$ is an equivalence.
                Analogously, any fully faithful triangulated functor $G: \dc{\mathsf{A}} \to \dc{\mathsf{A}}$ with $H^0(\mathsf{A}) \subset \essim(G)$ is an equivalence.
            In particular, in this situation $G_{\mid \tr(\mathsf{A})}:\tr(\mathsf{A})\to G(\tr(\mathsf{A}))$ is an equivalence and $G$ is its idempotent extension up to natural isomorphism by \cite[Theorem 1.5]{balmerschlichting}. 
            
            If $H^0(\mathsf{A})$ is a full and essentially wide\footnote{i.e. it contains at least one object for each isomorphism class.} subcategory of $\essim(G_{\mid H^0(\mathsf{A})})$, the inclusion $\tr(\mathsf{A}) \subset \essim(G_{\mid \tr(\mathsf{A})})$ obtained by Lemma \ref{lem:dgtaylor} is in fact an equivalence. This follows by considering the fully faithful triangulated functor
            \[
            \begin{tikzcd}
                L: \tr(\mathsf{A}) \ar[r,hook, "\incl"] & \essim(G_{\mid \tr (\mathsf{A})}) \ar[r,"(G_{\mid \tr(\mathsf{A})})^{-1}"] & \tr(\mathsf{A}),
            \end{tikzcd}    
            \]
            which is an equivalence since $H^0(\mathsf{A}) \subset \essim(L)$ by assumption. Therefore, since the second functor is already an equivalence, $\incl$ is an equivalence as well. In particular, we obtain an induced functor $G' : \tr(\mathsf{A}) \to \tr(\mathsf{A})$ whose idempotent extension is $G$, since $G'$ is defined as $L^{-1}=(\incl)^{-1} G_{\mid \tr (\mathsf{A})}$.
            \end{rem}

            \begin{prop} \label{prop:dcatra}
                Let $\mathsf{A}$ be a DG-category.
                If $\dc{\mathsf{A}}$ has a strongly unique enhancement, then $\tr(\mathsf{A})$ has a strongly unique enhancement.
            \end{prop}
            \begin{proof}    
                Take a triangulated equivalence $F:\tr(\mathsf{A})\to \tr(\mathsf{A})$. Now consider its idempotent extension $F':\dc {\mathsf{A}} \to \dc {\mathsf{A}}$, which is unique (up to natural isomorphism) by \cite[Theorem 1.5]{balmerschlichting}. Take $(\mathsf{C},E)$ an enhancement of $\dc{\mathsf{A}}$. Then we have a quasi-functor $\mathsf{f}':\mathsf{C}\to \mathsf{C}$ which is a $(\mathsf{C},E)$-lift of $F'$. We restrict $\mathsf{f}'$ to $\mathsf{C}_{\mid \tr(\mathsf{A})}$ (see Notation \ref{nota:restrictdg}). By the fact that $F'_{\mid \tr(\mathsf{A})}$ is $F$, the restiction of $\mathsf{f}'$ gives a quasi-functor $\mathsf{f}: \mathsf{C}_{\mid \tr(\mathsf{A})}\to \mathsf{C}_{\mid \tr(\mathsf{A})}$ (cf. Lemma \ref{lem:restrictdg}). We conclude that $\mathsf{f}$ is a $(\mathsf{C}_{\mid \tr(\mathsf{A})}, E_{\mid \tr(\mathsf{A})})$-lift of $F$.
            
                It remains to show that $\tr(\mathsf{A})$ has a unique enhancement, so that we can apply Proposition \ref{prop:strongenh} to conclude. Let $(\mathsf{D},E)$ and $(\mathsf{D}',E')$ be two enhancements of $\tr(\mathsf{A})$. Then $\perf(\mathsf{D})$ and $\perf(\mathsf{D}')$ are enhancements of $\dc{\mathsf{A}}$. Indeed, $\tr(\mathsf{D}) \cong H^0(\mathsf{D})\cong \tr(\mathsf{A})$, which implies $\dc{\mathsf{D}} \cong \dc{\mathsf{A}}$ by \cite[Theorem 1.5]{balmerschlichting}.
                By hypothesis, we have a quasi-equivalence quasi-functor  $\mathsf{g}: \perf(\mathsf{D})\to \perf (\mathsf{D}')$ lifting the identity of $\dc{\mathsf{A}}$. We now consider the restrictions $\perf(\mathsf{D})_{\mid \tr(\mathsf{A})}$ and $\perf(\mathsf{D}')_{\mid \tr(\mathsf{A})}$. Since $\mathsf{g}$ is a lift of the identity, it induces a quasi-equivalence quasi-functor $\mathsf{g}': \perf(\mathsf{D})_{\mid \tr(\mathsf{A})} \to \perf(\mathsf{D}')_{\mid \tr(\mathsf{A})}$. We conclude by the following diagram:
                \[
                \begin{tikzcd}
                    \mathsf{D} \ar[r,hook,"\cong"] & \perf(\mathsf{D})_{\mid \tr(\mathsf{A})}\ar[r,"\mathsf{g}'"] &\perf(\mathsf{D}')_{\mid \tr(\mathsf{A})} & \mathsf{D}' . \ar[l, hook', "\cong" above]
                \end{tikzcd}    
                \]
            \end{proof}
            \begin{rem}\label{rem:dcatrasemi}
                The previous result holds also if we replace strongly unique enhancement with semi-strongly unique enhancement. The only difference in the proof is that we need to consider semilifts.
            \end{rem}
            \begin{prop}\label{prop:uniquetrdc}
Let $\mathsf{A}$ be a DG-category. 
\begin{enumerate}
    \item If $\tr(\mathsf{A})$ has a unique enhancement, then $\dc{\mathsf{A}}$ has a unique enhancement.
    \item If $\dc{\mathsf{A}}$ has a unique enhancement, then $\derdg(\mathsf{A})$ has a unique enhancement.
\end{enumerate}
            \end{prop}
\begin{proof}
    \item 
    \begin{enumerate}
        \item Let $\mathsf{C}$ and $\mathsf{C}'$ be two enhancements of $\dc{\mathsf{A}}$. Then $\mathsf{C}_{\mid \tr(\mathsf{A})}$ and $\mathsf{C}'_{\mid \tr (\mathsf{A})}$ are enhancements of $\tr(\mathsf{A})$. In particular, there exists a quasi-equivalence quasi-functor $\mathsf{f}: \mathsf{C}_{\mid \tr(\mathsf{A})} \to \mathsf{C}'_{\mid \tr (\mathsf{A})}$. Consider $\mathsf{g}$ as the composition
        \[
        \begin{tikzcd}
            \mathsf{C} \ar[r,"\mathsf{y}", hook] & \perf (\mathsf{C}) & \perf (\mathsf{C}_{\mid \tr{\mathsf{A}}}) \ar[l,hook', "\perf (\incl)"] \ar[r,"\perf(\mathsf{f})"]  & \perf (\mathsf{C}'_{\mid \tr (\mathsf{A})})\ar[r,hook,"\perf(\incl)"] & \perf (\mathsf{C}') & \mathsf{C}'. \ar[l, hook', "\mathsf{y}"]
        \end{tikzcd}
        \]
        Lemma \ref{lem:dgtaylor} shows that $\mathsf{y}$ and $\perf(\incl)$ are quasi-equivalences for both $\mathsf{C}$ and $\mathsf{C}'$. Since $\perf(\mathsf{f})$ is a quasi-equivalence as well, $\mathsf{g}$ becomes a quasi-equivalence. 

    \item The proof is very similar to item 1: consider $\mathsf{C}$ and $\mathsf{C}'$ two enhancements of $\derdg(\mathsf{A})$. By hypothesis, there exists a quasi-equivalence quasi-functor $\mathsf{f}: \mathsf{C}_{\mid \dc{\mathsf{A}}}\to \mathsf{C}'_{\mid \dc {\mathsf{A}}}$. Let $\mathsf{h}$ be the composition
    \[
        \begin{tikzcd}
            \mathsf{C} \ar[r,"\phi"] & \semifree (\mathsf{C}_{\mid \dc{\mathsf{A}}} )\ar[r,"\semifree(\mathsf{f})"] & \semifree(\mathsf{C}'_{\mid \dc{\mathsf{A}}}) & \mathsf{C}' \ar[l, "\phi'" above],
        \end{tikzcd}
    \] 
    where $\phi,\phi'$ are defined as in \cite[Section 1]{luntsorlov}. By \cite[Proposition 1.17]{luntsorlov}, they both are quasi-equivalences. Since $\mathsf{f}$ is a quasi-equivalence, the same holds for $\semifree(\mathsf{f})$ (see \cite[Theorem 10.12.5.1]{bernlunts}, or \cite[Example 7.2]{k1}). Finally, $\mathsf{h}$ is a quasi-equivalence.
    \end{enumerate}
\end{proof}
            \section{Triangulated formality}\label{sec:sf}
            
            In this section we define the original notion of triangulated formal DG-categories and discuss their relation with uniqueness of enhancements (see Figure \ref{fig:sf} for an overview). This concept is inspired by the property of intrinsically formal graded rings proved in Proposition \ref{prop:unique} below. 

            \begin{nota}
                Any DG-ring $A$, i.e. a DG-algebra over a commutative ring $\ring$, can be thought of as a DG-category with one object. By a slight abuse of notation, $A$ will indicate both the DG-ring and the DG-category. Its unique object will be denoted with $O_A$.
            \end{nota}

            \begin{defn}\label{def:intr}
                We recall that a DG-morphism $f:A \to A'$ (i.e. a DG-functor between DG-rings) is a \emph{quasi-isomorphism} if $H^*(f)$ is an isomorphism.

                Let $B$ be a graded ring. We say that $B$ is \emph{intrinsically formal} if, for every DG-ring $A$ such that $H^*(A) \cong B$, we are able to find a zig-zag of quasi-isomorphisms
                \[
                \begin{tikzcd}
                    & A_1 \arrow{ld}[sloped,above]{\simeq}\ar[rd,"\simeq" sloped] && \dots \arrow{ld}[sloped,above]{\simeq}\ar[rd,"\simeq" sloped]& \\
                    A&&A_2&&B.
                \end{tikzcd}
                \]
            \end{defn}
            This concept has important connections with geometry, and comes from the study of formal manifolds~\cite{delignerealhomotopy}. 
    An explicit discussion of intrinsic formality in geometry, together with a number of examples, is the focus of \cite{lupton} (see also the PhD thesis~\cite{luptonphd}). 
    More recently, this notion has been considered in~\cite{saito23}, since it plays an important role in tilting theory for periodic triangulated categories. 
    This role will reappear in our studies in Section \ref{sec:freegen}, where we will discuss examples of triangulated categories coming from intrinsically formal graded rings (see Corollary \ref{cor:stronguniquefree} and Example \ref{es:Kmod}).
            \begin{prop}\label{prop:unique}
                Let $B$ be an intrinsically formal graded ring. For any enhancement $(\mathsf{C},E)$ of $\tr(B)$, there exists a quasi-equivalence quasi-functor $\mathsf{f}: B^{\pretr} \to \mathsf{C}$ such that $EH^0(\mathsf{f})(O_B)\cong O_B$.
            \end{prop}
            \begin{proof}
                Without loss of generality, we assume that $\mathsf{C}$ is strongly pretriangulated.
                Consider $C \in \mathsf{C}$ such that $E(C) \cong O_B$. Since 
            \begin{equation*}
            \begin{split}
            H^n(\hom_{\mathsf{C}}(C,C)) &\cong H^0 (\hom_{\mathsf{C}}(C,C)[n])\cong
            H^0(\hom_{\mathsf{C}}(C,C[n]))\\
            &\cong \hom_{H^0(\mathsf{C})}(C,C[n])\cong\hom_{\tr(B)}(O_B,O_B [n]) \cong H^n(B)=B^n,
            \end{split}
            \end{equation*} $\hom_{\mathsf{C}}(C,C)$ has cohomology $B$ (the product of $B$ is recovered from the composition on $\tr(B)$). By hypothesis, we obtain a zig-zag of quasi-isomorphisms from $B$ to $\hom_{\mathsf{C}}(C,C)$ extending to a quasi-fully faithful quasi-functor $\mathsf{f}:B^{\pretr} \to \mathsf{C}$ by Remark \ref{rem:h0qespretr} and Proposition \ref{prop:dgextension}. Since $E H^0(\mathsf{f})(O_B)\cong O_B$, Remark \ref{rem:dgtaylor} shows that $H^0(\mathsf{f})$ is an equivalence, so $\mathsf{f}$ is a quasi-equivalence.
            \end{proof}

            The previous proposition motivates the following. 
            \begin{defn}
                A DG-category $\mathsf{A}$ is \emph{triangulated formal} if
                \begin{description}
                    \item[TF] For any enhancement $(\mathsf{C},E)$ of $\tr(\mathsf{A})$, we have a quasi-equivalence quasi-functor $\mathsf{f}: \mathsf{A}^{\pretr} \to \mathsf{C}$ such that \begin{equation}\label{eq:sf2}
                        E H^0(\mathsf{f}) (X) \cong X\qquad \text{ for all }X \in H^0(\mathsf{A}).\end{equation}
                \end{description}
                A DG-category $\mathsf{A}$ is \emph{unbounded triangulated formal} if
                \begin{description}
                    \item[uTF]  Given any enhancement $(\mathsf{C},E)$ of $\derdg(\mathsf{A})$, there exists a quasi-equivalence quasi-functor $\mathsf{f}: \semifree(\mathsf{A}) \to \mathsf{C}$ such that \eqref{eq:sf2} holds.
                \end{description}
            \end{defn}
            This notion is thus a generalization of the intrinsic formality to the case of DG-categories;  here, ``triangulated'' reminds the reader of the strong connection between this notion and some enhancement properties of the associated triangulated categories, as will be explored extensively in this section (see Figure \ref{fig:sf}).

            \begin{rem}\label{rem:sfunique}Let $\mathsf{A}$ a DG-category.
                \begin{enumerate}
                \item Triangulated formality is stable under quasi-equivalence.
                \item If $\tr(\mathsf{A})\cong \tr(\mathsf{A}')$ via the inclusion $\mathsf{A}\subset \mathsf{A}'$, then $\mathsf{A}$ is (unbounded) triangulated formal whenever $\mathsf{A}'$ is (unbounded) triangulated formal.
                    \item   If $\mathsf{A}$ is triangulated formal, then $\tr(\mathsf{A})$ has a unique enhancement. Analogously, if $\mathsf{A}$ is unbounded triangulated formal, then $\derdg(\mathsf{A})$ has a unique enhancement.
                    \item If $\tr(\mathsf{A})$ has a semi-strongly unique enhancement, then $\mathsf{A}$ is triangulated formal by picking a $(\mathsf{A}^{\pretr},\id)- (\mathsf{C},E)$-semilift of the identity $\id:\tr(\mathsf{A})\to \tr(\mathsf{A})$ for any enhancement $(\mathsf{C},E)$ (we use Proposition \ref{prop:semistrongenh}). Analogously, if $\derdg(\mathsf{A})$ has a semi-strongly unique enhancement, then $\mathsf{A}$ is unbounded triangulated formal.
                    \item As a matter of fact, $\tr(\mathsf{A})$ has a semi-strongly unique enhancement if and only if $\mathsf{A}^{\pretr}$ is triangulated formal.
                \end{enumerate}
            \end{rem}

            \begin{es}\label{es:intrformalstrong}
                Intrinsically formal graded rings are examples of triangulated formal graded categories with one object by Proposition \ref{prop:unique}. In particular, all rings are triangulated formal (see, for instance, \cite[Lemma 6.6]{duggershipleyK}).
            \end{es}

            We now provide a wide range of meaningful examples of triangulated formal DG-categories.
            \begin{prop}\label{prop:asf}
                Let $\mathcal{A}$ be a ($\ring$-linear) category. Then it is triangulated formal (cf. \cite[Proposition 2.6]{luntsorlov}). 
                \end{prop}
                \begin{proof}
                The idea is to proceed analogously to the proof of a ring being intrinsically formal (cf. \cite[Lemma 6.6]{duggershipleyK}).
                Let $(\mathsf{C},E)$ be an enhancement of $\tr(\mathcal{A})$. For the sake of simplicity, assume $\mathsf{C}$ to be strongly pretriangulated and consider $\mathsf{C}_{\mid \mathcal{A}}$ (recall Notation \ref{nota:restrictdg}). Take $\tau_{\le 0}(\mathsf{C}_{\mid \mathcal{A}})$ as described in Definition \ref{def:truncation}. We have two natural DG-functors given by truncation: $\tau_{\le 0}(\mathsf{C}_{\mid \mathcal{A}}) \to \mathsf{C}_{\mid \mathcal{A}}$ and, by Remark \ref{rem:lefttruncation}, $\tau_{\le 0}(\mathsf{C}_{\mid \mathcal{A}}) \to H^0(\mathsf{C}_{\mid \mathcal{A}}){\cong} \mathcal{A}$ (the equivalence $H^0(\mathsf{C}_{\mid \mathcal{A}}){\cong} \mathcal{A}$ is given by the restriction of $E$). It is easy to prove that these DG-functors are in fact quasi-equivalences, because $H^i(\mathsf{C}_{\mid \mathcal{A}})=0$ for $i\ne 0$. By Remark \ref{rem:h0qespretr} and Proposition \ref{prop:dgextension}, we can extend the zig-zag $\mathcal{A} \leftarrow \tau_{\le 0} (\mathsf{C}_{\mid \mathcal{A}}) \to \mathsf{C}_{\mid \mathcal{A}}$ to obtain a quasi-equivalence quasi-functor $\mathsf{f}:\mathcal{A}^{\pretr} \to \mathsf{C}$. The fact that $EH^0(\mathsf{f})(X) \cong X$ for all $X \in \mathcal{A}$ follows from the definition of $\mathsf{f}$; indeed, $H^0(\mathsf{f})$ restricted to $\mathcal{A}$ is the inverse of $E$ on objects.
                \end{proof}
                
                \begin{prop}\label{prop:kcomba}
                Let $\mathcal{A}$ be an additive category. Then $\kcomb(\mathcal{A})$ has a semi-strongly unique enhancement.
                \end{prop}
                \begin{proof}
                We recall that $\kcomb(\mathcal{A})\cong\tr(\mathcal{A})$ by Example \ref{es:homder}, so $\kcomb(\mathcal{A})$ has a unique enhancement from Proposition \ref{prop:asf} and Remark \ref{rem:sfunique}. To simplify the notation, assume $\kcomb(\mathcal{A}) = \tr(\mathcal{A})$.
                We want to show item 3 of Proposition \ref{prop:semistrongenh} for the enhancement $(\mathcal{A}^{\pretr}, \id)$. Let $F$ be an autoequivalence of $\kcomb(\mathcal{A})$. Since $\mathcal{A}$ is triangulated formal by Proposition \ref{prop:asf}, considering the enhancement $(\mathcal{A}^{\pretr},F)$ of $\kcomb(\mathcal{A})$, we get a quasi-equivalence quasi-functor $\mathsf{f}:\mathcal{A}^{\pretr} \to \mathcal{A}^{\pretr}$ such that $F H^0(\mathsf{f}) (X) \cong X$ for all $X \in \mathcal{A}$.
                Let $G:= F H^0(\mathsf{f})$. Then $G_{\mid \mathcal{A}}$ gives an equivalence $\mathcal{A} \to \mathcal{A}$, so we can consider the DG-functor $\mathsf{g}:= (G_{\mid \mathcal{A}})^{\pretr}: \mathcal{A}^{\pretr} \to \mathcal{A}^{\pretr}$. Of course, $GH^0(\mathsf{g}^{-1})$ is the identity when restricted to $\mathcal{A}$. By \cite[Proposition 3.2]{chenye}, $GH^0(\mathsf{g}^{-1}) (X) \cong X$ for all $X \in \kcomb(\mathcal{A})$. By recalling the definition of $G$ and moving the equivalences around, we get $F(X) \cong H^0(\mathsf{g}\mathsf{f}^{-1})(X)$ for all $X \in \kcomb(\mathcal{A})$. This is exactly item 3 of Proposition \ref{prop:semistrongenh}, as wanted.
                \end{proof}
                
                Proposition \ref{prop:realdg} below is inspired by \cite[Theorem 3.2]{kelvos}. In that article, the authors used the definition of algebraic triangulated categories via Frobenius categories (see \cite[Proposition 3.1]{canonacostellari}). Using DG-categories, we are able to say something more, and provide a proof of uniqueness of enhancements for bounded derived categories of exact categories.\footnote{This is already known in the context of stable $\infty$-categories: see \cite[Corollary 7.59]{bunke2019controlled}.} Furthermore, the DG-category $\mathcal{E}_{\dg}$ associated is immediately triangulated formal.
                
                    \begin{prop} \label{prop:realdg}
                        Let $\mathcal{T}$ be an algebraic triangulated category and let $\mathcal{E}$ be an \emph{admissible exact subcategory}, i.e. an extension closed subcategory $\mathcal{E}$ of $\mathcal{T}$ such that $\hom(X,Y[n])=0$ for $n<0$ (this has an induced exact structure by \cite{dyer}). 
                        
                        Then, for any enhancement $(\mathsf{C},E)$ of $\mathcal{T}$, there exists a realization functor $\operatorname{real}: \derb(\mathcal{E}) \to \mathcal{T}$ admitting a $(\derb_{\dg} (\mathcal{E}),\id )-(\mathsf{C},E)$-lift,\footnote{For the sake of simplicity, we assume $\derb(\mathcal{E})= H^0(\derb_{\dg}(\mathcal{E}))$.} where $\derb_{\dg}(\mathcal{E})$ was defined in Example \ref{es:homder}. 
                    \end{prop}
                \begin{proof}
                    As our reasoning will not be affected by the quasi-equivalence inclusion $\mathsf{y}:\mathsf{C} \to \mathsf{C}^{\pretr}$, for the sake of simplicity we assume $\mathsf{C}$ to be strongly pretriangulated, and consider the DG-functor $(\mathsf{C}_{\mid \mathcal{E}})^{\pretr} \to \mathsf{C}$ obtained by Proposition \ref{prop:dgextension}.
                
                   From the natural truncation $\tau_{\le 0} \mathsf{C}_{\mid \mathcal{E}} \to \mathsf{C}_{\mid \mathcal{E}}$, and the quasi-equivalence $\tau_{\le 0}\mathsf{C}_{\mid \mathcal{E}} \to H^0(\mathsf{C}_{\mid \mathcal{E}})\cong \mathcal{E}$ (this is a quasi-equivalence because $\mathcal{E}\subset \mathcal{T}$ is admissible by assumption), Proposition \ref{prop:dgextension} gives rise to a quasi-functor $\mathsf{f}: \mathcal{E}^{\pretr} \to (\mathsf{C}_{\mid \mathcal{E}})^{\pretr} \to \mathsf{C}$. At the homotopy level, this defines a triangulated functor $\kcomb(\mathcal{E}) \to \mathcal{T}$ (recall Example \ref{es:homder}).
                
                   We now want to prove that $\acb(\mathcal{E}) \to \kcomb(\mathcal{E}) \to \mathcal{T}$ is the zero functor. Indeed, given any conflation $0 \to A \to B \to C \to 0$ in $\mathcal{E}$, we obtain a commutative diagram
                    \[
                    \begin{tikzcd}
                     A\ar[d,equal] \ar[r,"f"] & B \ar[r]\ar[d,equal] & \cone (f)\ar[r] \ar[d,"\cong"]& A[1]\ar[d,equal]\\
                     A \ar[r,"f"] & B \ar[r] & C\ar[r] & A[1]
                    \end{tikzcd}
                    \]
                    of distinguished triangles in $\mathcal{T}$.
                    Since $\hom(A[1], C)=\hom(A,C[-1])=0$, the morphism $\cone (f) \to C $ is determined by $B \to \cone (f) \to C$, which is exactly the map appearing in the conflation. Looking at $\cone (f) \to C$ in $\kcomb(\mathcal{E})$, its cone is the conflation $0\to A \to B \to C\to 0$, and its image is zero since $\cone(f) \to C$ is an isomorphism in $\mathcal{T}$. This shows that conflations are sent to zero via $\kcomb(\mathcal{E}) \to \mathcal{T}$. By Lemma \ref{lem:genconfl}, we conclude that $\acb(\mathcal{E}) \to \kcomb(\mathcal{E}) \to \mathcal{T}$ is the zero functor.
                
                    By Remark \ref{rem:0qf}, at the DG-level we have that $\acb_{\dg}(\mathcal{E}) \to \comdgb{\mathcal{E}} \cong \mathcal{E}^{\pretr} \to \mathsf{C}$ is the trivial quasi-functor. Therefore, we obtain an induced quasi-functor $\mathsf{r}:\derb_{\dg}(\mathcal{E}) \to \mathsf{C}$ satisfying the statement.
                \end{proof}
                \begin{cor}\label{cor:dbunique}
                    For any exact category $\mathcal{E}$, $\derb(\mathcal{E})$ has a unique enhancement (cf. \cite[Corollary 7.59]{bunke2019controlled}). More precisely, the full DG-subcategory $\mathcal{E}_{\dg}:= \derb_{\dg}(\mathcal{E})_{\mid \mathcal{E}}$ is triangulated formal. 
                    
                    In addition, given any enhancement $(\mathsf{C},E)$ of $\derb(\mathcal{E})$, the pretriangulated closure of the truncation $\mathsf{p}_{\le 0}: \tau_{\le 0}\mathsf{C}_{\mid \mathcal{E}} \to \mathsf{C}_{\mid \mathcal{E}}$ gives rise to a DG-quotient.
                \end{cor}
                \begin{proof}
                    Notice that $\tr(\mathcal{E}_{\dg})\cong \derb(\mathcal{E})$ since $\tr(\mathcal{E}_{\dg}) \subset H^0(\derb_{\dg}(\mathcal{E}))\cong \derb(\mathcal{E})$ and $\tr(\mathcal{E}_{\dg})$ is the triangulated envelope of $\mathcal{E}$. By applying Proposition \ref{prop:realdg} to $\mathcal{T}=\derb(\mathcal{E})$, we can construct a quasi-equivalence quasi-functor between any enhancement of $\derb(\mathcal{E})$ and $\derb_{\dg}(\mathcal{E})$ (the quasi-functor is a quasi-equivalence by 
                    \cite[Corollary A.7.1 and Proposition A.7]{positselski} and \cite[Corollary 2.8]{chenhanzhou}, which holds also for exact categories; more details are given in \cite{lorenzinphd}). Moreover, this quasi-equivalence fixes $\mathcal{E}$, so \eqref{eq:sf2} is satisfied. The last part of the statement follows from the construction of the realization functor in the proof of Proposition \ref{prop:realdg}.
                \end{proof}
                
                \begin{prop}\label{prop:derbssu}
                Let $\mathcal{E}$ be an exact category. Then $\derb(\mathcal{E})$ has a semi-strongly unique enhancement.
                \end{prop}
                \begin{proof}
                    Mimicking the proof of Proposition \ref{prop:kcomba} with the natural enhancement obtained by $\derb_{\dg}(\mathcal{E})$, the statement follows since the reasoning of \cite[Proposition 3.2 and Proposition 3.7]{chenye} can be adapted to this setting. The only detail to be precise about is the fact that $G_{\mid \mathcal{E}}$ gives an exact equivalence $\mathcal{E}\to \mathcal{E}$, which gives a DG-functor $\comdgb{\mathcal{E}} \to \comdgb{\mathcal{E}}$ inducing a quasi-functor $\mathsf{g}:\derb_{\dg}(\mathcal{E}) \to \derb_{\dg}(\mathcal{E})$ via the property of DG-quotients.
                \end{proof}
                \begin{rem}
                    Since additive categories and abelian categories are examples of exact categories, Proposition \ref{prop:derbssu} in fact generalizes both Proposition \ref{prop:kcomba} and the bounded case of \cite[Remark 5.4]{canonaco2021uniqueness}, which shows semi-strong uniqueness of enhancements for the derived categories of abelian categories under any boundedness requirement.    
                \end{rem}
                Let us state some results relating triangulated formality with the uniqueness of enhancements.
                First of all, we motivate why we avoided the notion of triangulated formality for perfect complexes.
            \begin{prop} \label{prop:perfectsf2}
                A DG-category $\mathsf{A}$ is triangulated formal if and only if the following holds
                \begin{description}
                    \item[cTF] For any enhancement $(\mathsf{C},E)$ of $\dc{\mathsf{A}}$, we can choose a quasi-equivalence quasi-functor $\mathsf{f}:\perf(\mathsf{A})\to \mathsf{C}$ satisfying \eqref{eq:sf2}.
                \end{description}
            \end{prop}
            \begin{proof}
                TF $\Rightarrow$ cTF. Let $(\mathsf{C},E)$ be an enhancement of $\dc{\mathsf{A}}$ and consider the restriction $\mathsf{C}_{\mid \tr(\mathsf{A})}$ (recall Notation \ref{nota:restrictdg}). By triangulated formality, we get a quasi-equivalence quasi-functor $\mathsf{f}':\mathsf{A}^{\pretr} \to \mathsf{C}_{\mid \tr(\mathsf{A})}$ satisfying \eqref{eq:sf2}.
                Denote by $\mathsf{f}$ the following composition:
                \[
                \begin{tikzcd}
                \perf(\mathsf{A}) \ar[r,hook,"\perf(\mathsf{y})"]&\perf(\mathsf{A}^{\pretr}) \ar[r,"\perf({\mathsf{f}'})"]&\perf(\mathsf{C}_{\mid \tr(A)}) \ar[r,hook,"\perf(\incl)"]& \perf (\mathsf{C})& \mathsf{C}\ar[l,hook',"\mathsf{y}" above]
                \end{tikzcd}
                \]
                Notice that $EH^0(\mathsf{f})(X)\cong X$ for all $X \in H^0(\mathsf{A})$, and all DG-functors considered in the composition are quasi-fully faithful. Lemma \ref{lem:dgtaylor} shows that $\mathsf{f}$ is a quasi-equivalence.

                cTF $\Rightarrow$ TF. Given any enhancement $(\mathsf{D},F)$ of $\tr(\mathsf{A})$, then $\perf(\mathsf{D})$ is an enhancement of $\dc{\mathsf{A}}$ with the unique extension of $F$ (we recall Remark \ref{rem:dc}). Moreover, $\mathsf{D} \cong \perf(\mathsf{D})_{\mid \tr(\mathsf{A})}$ via inclusion. By cTF, we have a quasi-equivalence quasi-functor $\mathsf{g}: \perf (\mathsf{A}) \to \perf (\mathsf{D})$ satisfying \eqref{eq:sf2}. Restricting $\mathsf{g}$ to $\mathsf{A}^{\pretr}$, by Remark \ref{rem:dgtaylor} we get a quasi-equivalence $\mathsf{A}^{\pretr} \to \perf(\mathsf{D})_{\mid \tr(\mathsf{A})}$ satisfying \eqref{eq:sf2}.\end{proof}
            
            \begin{rem}\label{}
            From Proposition \ref{prop:perfectsf2}, $\dc{\mathsf{A}}$ has a unique enhancement for any triangulated formal  DG-category $\mathsf{A}$.
            \end{rem}
            \begin{prop}\label{prop:usf}
                A triangulated formal DG-category $\mathsf{A}$ is also unbounded triangulated formal.
            \end{prop}
            \begin{proof}
                Let $(\mathsf{C},E)$ be an enhancement of $\derdg(\mathsf{A})$ and consider $\mathsf{C}':= \mathsf{C}_{\mid \dc{\mathsf{A}}}$. By Proposition \ref{prop:perfectsf2}, we obtain a quasi-equivalence quasi-functor $\mathsf{f}: \perf(\mathsf{A}) \to \mathsf{C}'$ satisfying \eqref{eq:sf2}. Let us define $\mathsf{h}$ as the composition  
                \[
                \begin{tikzcd}
                    \semifree(\mathsf{A}) \ar[r,"\phi_{\perf(\mathsf{A})}"] & \semifree(\perf(\mathsf{A}))\ar[r,"\semifree(\mathsf{f})"] & \semifree(\mathsf{C}') & \mathsf{C},\ar[l,"\phi_{\mathsf{C}'}" above]
                \end{tikzcd}    
                \]
                where $\phi_{\perf(\mathsf{A})}$ and $\phi_{\mathsf{C}'}$ are quasi-functors described in \cite[Section 1]{luntsorlov}. Then $\mathsf{h}$ is a quasi-equivalence as explained in the proof of item 2 of Proposition \ref{prop:uniquetrdc}. 

                We are reduced to check that $\eqref{eq:sf2}$ is satisfied. In \cite{luntsorlov}, $\phi_{\perf(\mathsf{A})}$ and $\phi_{\mathsf{C}'}$
                are obtained by a Yoneda embedding and a restriction functor, both of which do not affect the subcategory associated ($\perf(\mathsf{A})$ and $\mathsf{C}'$ respectively). Therefore, since $\semifree(\mathsf{f})$ is an extension of $\mathsf{f}$, $\mathsf{h}$ fulfils \eqref{eq:sf2}.
            \end{proof}
We can prove the converse of Proposition \ref{prop:usf} in a special case.
\begin{prop}\label{prop:eeusf}
    An unbounded triangulated formal DG-category $\mathsf{A}$ is also triangulated formal if the following holds:
    \begin{description}
        \item[EE] For any enhancement $(\mathsf{C},E)$ of $\tr(\mathsf{A})$, $E$ extends to a triangulated equivalence $E': \derdg(\mathsf{C}) \to \derdg(\mathsf{A})$ up to natural isomorphism, i.e. $E'_{\mid H^0(\mathsf{C})}\cong E$.
    \end{description}
    \emph{(EE stands for Extending Enhancements).}
\end{prop}
\begin{proof}
    Let $(\mathsf{C},E)$ be an enhancement of $\tr(\mathsf{A})$. By assumption, we can consider the enhancement $(\semifree(\mathsf{C}),E')$ of $\derdg(\mathsf{A})$ such that $E'_{\mid H^0(\mathsf{C})}\cong E$. 
    Since $\mathsf{A}$ is unbounded triangulated formal, there exists $\mathsf{f}: \semifree(\mathsf{A}) \to \semifree(\mathsf{C})$ such that $E'H^0(\mathsf{f})(X) \cong X$ for all $X \in H^0(\mathsf{A})$.  

    We now want to show that $\essim(\mathsf{f}_{\mid \mathsf{A}^{\pretr}})$ is contained in $\bar{\mathsf{C}}$, the homotopy closure of $\mathsf{C}$ in $\semifree(\mathsf{C})$. Let $Y \in \essim(\mathsf{f}_{\mid \mathsf{A}})$. Then there exists $X \in H^0(\mathsf{A})$ such that $Y \cong H^0(\mathsf{f})(X)$. By applying $E'$, we have $E'Y \cong E' H^0(\mathsf{f})(X) \cong X$. Since $E'_{\mid H^0(\mathsf{C})} \cong E$, we can choose $Y' \in H^0(\mathsf{C})$ such that $E'Y' \cong X \cong E'Y$. In particular, $Y$ is homotopy equivalent to $Y'$. From this, we have $\mathsf{f}_{\mid \mathsf{A}}: \mathsf{A} \to \bar{\mathsf{C}}$. Being $H^0(\bar{\mathsf{C}})$ triangulated,
    $\essim(\mathsf{f}_{\mid \mathsf{A}^{\pretr}})\subseteq \bar{\mathsf{C}}$, as wanted.
    
    Consider the quasi-functor $\mathsf{g}: \mathsf{A}^{\pretr} \to \bar{\mathsf{C}} \hookleftarrow \mathsf{C}$, where the first map is the quasi-functor $\mathsf{f}_{\mid \mathsf{A}^{\pretr}}$ and the second is a Yoneda embedding. By Remark \ref{rem:dgtaylor}, we conclude that $\mathsf{g}$ is a quasi-equivalence, so $\mathsf{A}$ is triangulated formal.
\end{proof}
\begin{rem}\label{rem:ssee}
    Assume $\mathsf{A}$ is a DG-category for which EE holds and $\derdg(\mathsf{A})$ has a semi-strongly unique enhancement. Then $\perf(\mathsf{A})$ is unbounded triangulated formal because $\semifree(\perf (\mathsf{A}))\cong \semifree(\mathsf{A})$ by \cite[Proposition 1.17]{luntsorlov}. Since $\mathsf{A}$ satisfies EE, from \cite[Theorem 1.5]{balmerschlichting} one can prove that $\perf(\mathsf{A})$ also satisfies EE. By Proposition \ref{prop:eeusf}, $\perf(\mathsf{A})$ is triangulated formal, meaning that $\dc{\mathsf{A}}$ has a semi-strongly unique enhancement by Remark \ref{rem:sfunique}. 
\end{rem}

\begin{figr}\label{fig:sf}
Relation between  triangulated formality and uniqueness of enhancements for a DG-category $\mathsf{A}$.
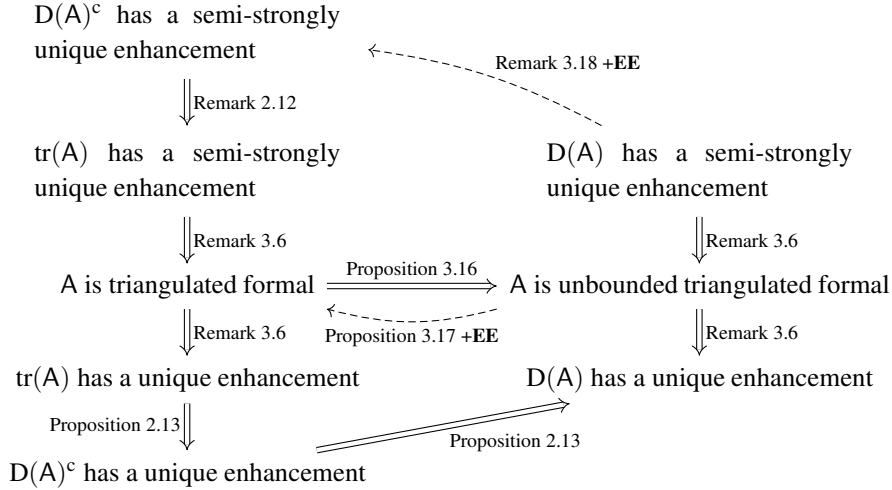
    \begin{center}
    \begin{tikzcd}[column sep=10ex]
    \begin{tabular}{p{0.3\textwidth}}$\dc{\mathsf{A}}${ has a semi-strongly unique enhancement}\end{tabular}\ar[d,Rightarrow,"\text{Remark }\ref{rem:dcatrasemi}"] & \\
    \begin{tabular}{p{0.3\textwidth}}$\tr(\mathsf{A})${ has a semi-strongly unique enhancement}\end{tabular}\ar[d,Rightarrow,"\text{Remark }\ref{rem:sfunique}" left] & \begin{tabular}{p{0.3\textwidth}}$\derdg(\mathsf{A})${ has a semi-strongly unique enhancement}\end{tabular}\ar[d,Rightarrow,"\text{Remark \ref{rem:sfunique}}"]\ar[lu,bend right=10,dashed, "\text{Remark \ref{rem:ssee} +}\textbf{EE}" above right]\\
    \mathsf{A} \text{ is triangulated formal}\ar[r,Rightarrow,"\text{Proposition \ref{prop:usf}}"] \ar[d,Rightarrow,"\text{Remark \ref{rem:sfunique}}"]\ar[u,dashed,bend right=30, "\text{Remark \ref{rem:sfunique} + }\mathsf{A}=\mathsf{A}^{\pretr}" right] & \mathsf{A} \text{ is unbounded triangulated formal}\ar[d,Rightarrow,"\text{Remark \ref{rem:sfunique}}"]\ar[l,bend left=15, dashed, "\text{Proposition \ref{prop:eeusf} +} \textbf{EE}", start anchor= south west, end anchor= south east]\\
    \tr(\mathsf{A})\text{ has a unique enhancement}\ar[d, Rightarrow, "\text{Proposition \ref{prop:uniquetrdc}}" left] & \derdg(\mathsf{A})\text{ has a unique enhancement}\\
    \dc{\mathsf{A}}\text{ has a unique enhancement} \ar[ru, Rightarrow, "\text{Proposition \ref{prop:uniquetrdc}}" below right]&
    \end{tikzcd}
    \end{center}
\end{figr}

\section{Formal standardness}\label{sec:fs}

\begin{prop}\label{prop:natfgrad}
    Let $\mathcal{T}$ be a triangulated category and consider $\mathcal{S}\subset \mathcal{T}$ a full subcategory.
    Then any  triangulated equivalence $F: \mathcal{T} \to \mathcal{T}$ such that $FX \cong X$ for $X \in \mathcal{S}$ is naturally isomorphic to a triangulated equivalence $G$ such that $GX=X$ for $X \in \mathcal{S}$.
\end{prop}
\begin{proof}
    Let us consider the family of isomorphisms $\eta:= (\eta_X)$, where $\eta_X : X \to FX$ is an isomorphism if $X \in \mathcal{S}$, while $\eta_X=\id$ if $X \notin \mathcal{S}$.\footnote{This definition explicitly requires the axiom of choice, since we choose an isomorphism for every object in $\mathcal{S}$.} Then $G_{X,Y}:= \eta_Y^{-1} F_{X,Y} \eta_X$ describes the wanted triangulated equivalence $G$, and $\eta$ becomes a natural isomorphism $G \to F$.
\end{proof}
            \begin{defn}\label{def:fgrad}
            Let $\mathsf{A}$ be a DG-category and consider a triangulated autoequivalence $(F,\eta)$ on $\tr(\mathsf{A})$ 
            such that 
            \begin{equation}\label{eq:fgrad}
                FX=X \quad \text{for any }X \in H^0(\mathsf{A}).      
            \end{equation}
            Its \emph{graded restriction} is a graded functor $F_{\mid H^*(\mathsf{A})}^{gr}:H^*(\mathsf{A}) \to H^*(\mathsf{A})$  defined by $F_{\mid H^*(\mathsf{A})}^{gr}(X)=X$ and 
            \[
                \begin{tikzcd}[column sep=large]
                    \hom_{H^*(\mathsf{A})}(X,Y) \ar[d,"\cong"]\ar[rr,"({F_{\mid H^*(\mathsf{A})}^{gr}})_{X,Y}"]&&\hom_{H^*(\mathsf{A})}(X,Y) \ar[d,"\cong"]\\
                    \displaystyle\bigoplus_{i} \hom_{\tr(\mathsf{A})}(X,Y[i]) \ar[r,"\bigoplus_i F_{X, Y[i]}"]& \displaystyle\bigoplus_{i} \hom_{\tr(\mathsf{A})}(FX,F(Y[i])) \arrow{r}{\bigoplus_i \eta_{ Y}^i} & \displaystyle\bigoplus_{i} \hom_{\tr(\mathsf{A})}(X,Y[i])
                \end{tikzcd}
            \]
            for any $X,Y \in \mathsf{A}$, where the vertical arrows are obtained via the following isomorphisms\footnote{The first and the last isomorphisms are given by definition, while the second one is obtained by iterated composition of the closed morphisms associated to the suspensions.} \[
                \hom_{H^*(\mathsf{A})}(X,Y)\cong \bigoplus_i H^i(\hom_{\mathsf{A}}(X,Y)) \cong \bigoplus_i H^0 (\hom_{\mathsf{A}}(X,Y[i]))
            \cong \bigoplus_{i} \hom_{H^0(\mathsf{A})}(X,Y[i]). \]
            \end{defn}
            \begin{defn}\label{def:formstand}
            A DG-category $\mathsf{A}$ is 
            \begin{itemize}
                \item \emph{formally standard}  
                if, given two triangulated equivalences $F,G : \tr(\mathsf{A}) \to \tr(\mathsf{A})$ 
                satisfying \eqref{eq:fgrad} and $F_{\mid H^*(\mathsf{A})}^{gr} \cong G_{\mid H^*(\mathsf{A})}^{gr}$, there is a natural isomorphism $F \cong G$.
                \item \emph{lifted} 
                if for every triangulated equivalence $F:\tr(\mathsf{A}) \to \tr(\mathsf{A})$ 
                for which \eqref{eq:fgrad} holds, we have a quasi-functor $\mathsf{f}:\mathsf{A} \to \mathsf{A}$ such that $H^*(\mathsf{f}) \cong F_{\mid H^*(\mathsf{A})}^{gr}$. 
            \end{itemize}
            \end{defn}
            The notion of lifted is introduced to treat at the same time the two following examples, which are crucial for applications.

            
\begin{es}\label{es:lifted}\item 
    \begin{itemize}
        \item Every graded category $\mathsf{B}$ is lifted: 
        indeed, given a triangulated equivalence $F: \tr(\mathsf{B})\to \tr(\mathsf{B})$, 
        the quasi-functor $\mathsf{f}$ required is simply given by the graded restriction $F^{gr}_{\mid \mathsf{B}}$. 
\item    Let $\mathcal{E}$ be an exact category and consider $\mathcal{E}_{\dg} := \derb_{\dg}(\mathcal{E})_{\mid \mathcal{E}}$ as in Corollary \ref{cor:dbunique}. We claim that $\mathcal{E}_{\dg}$ is lifted: indeed, for any triangulated equivalence $F:\tr(\mathcal{E}_{\dg})\to \tr(\mathcal{E}_{\dg})$ satisfying \eqref{eq:fgrad}, since $\tr(\mathcal{E}_{\dg}) \cong \derb(\mathcal{E})$, we obtain an exact equivalence ${\mathcal{E}} \to {\mathcal{E}}$ (see \cite[Proposition A.7]{positselski}). This induces an equivalence $\comdgb{\mathcal{E}} \to \comdgb{\mathcal{E}}$; via quotient we get a quasi-functor $\derb_{\dg}(\mathcal{E}) \to \derb_{\dg}(\mathcal{E})$, and finally a quasi-functor $\mathsf{f}:\mathcal{E}_{\dg}\to \mathcal{E}_{\dg}$. Moreover, the exact equivalence $\mathcal{E}\to \mathcal{E}$ uniquely determines what happens on $H^*(\mathcal{E}_{\dg})$; this is Proposition \ref{prop:0detgr}. 
We conclude that $H^*(\mathsf{f})\cong F^{gr}_{\mid H^*(\mathcal{E}_{\dg})}$. 
    \end{itemize}
\end{es}

            \begin{rem}\label{rem:gradrestid}
                A DG-category $\mathsf{A}$ is formally standard 
                if and only if any triangulated equivalence $F: \tr(\mathsf{A})\to \tr(\mathsf{A})$, 
                 such that \eqref{eq:fgrad} holds and $F_{\mid H^*(\mathsf{A})}^{gr}\cong \id $, is naturally isomorphic to the identity.
            \end{rem}
            \begin{rem}\label{rem:performstand}
                Replacing $\tr(\mathsf{A})$ with $\dc{\mathsf{A}}$ in Definition \ref{def:formstand}, one could be tempted to define \emph{perfect formally standard} and \emph{perfect lifted} DG-categories.
                However, this is not useful, since the equivalences $F:\dc{\mathsf{A}}\to \dc{\mathsf{A}}$ satisfying \eqref{eq:fgrad} are determined by $F_{\mid \tr(\mathsf{A})}$ up to natural isomorphism by \cite[Theorem 1.5]{balmerschlichting}, and Remark \ref{rem:dgtaylor} shows that $\tr(\mathsf{A})\subset \essim(F_{\mid \tr(\mathsf{A})})$, so that $F_{\mid \tr(\mathsf{A})}$ can be thought of as an equivalence $\tr(\mathsf{A})\to \tr(\mathsf{A})$. In other words, the only equivalences admitting a graded restriction on $\dc{\mathsf{A}}$ are equivalences restricting to $\tr(\mathsf{A})$.
            \end{rem}
            \section{Main result}\label{sec:maintheorem}
            In this section, we show how the new notions introduced so far (triangulated formality, formal standardness and liftedness) are connected to strong uniqueness of enhancements. In the case of graded categories $\mathsf{B}$, we are able to provide a characterization of strong uniqueness of enhancements for $\tr(\mathsf{B})$ and $\dc{\mathsf{B}}$ (see Theorem \ref{thm:main}). 
            
            Let us start with a sufficient condition for strong uniqueness of enhancements.
            \begin{prop}\label{prop:sffssue}
            Let $\mathsf{A}$ be a lifted, triangulated formal and formally standard DG-category. Then $\tr(\mathsf{A})$ has a strongly unique enhancement.
            \end{prop}
            \begin{proof}
                Let $F:\tr (\mathsf{A}) \to \tr(\mathsf{A})$ be any triangulated equivalence. Let us consider the enhancement $(\mathsf{A}^{\pretr}, F)$ of $\tr(\mathsf{A})$. By triangulated formality, there exists $\mathsf{f}:\mathsf{A}^{\pretr} \to \mathsf{A}^{\pretr}$ such that $FH^0(\mathsf{f})(X)\cong X$ for all $X \in H^0(\mathsf{A})$.
                Up to natural isomorphism, we can assume that $G:=FH^0(\mathsf{f})$ satisfies \eqref{eq:fgrad} by Proposition \ref{prop:natfgrad}.
We aim to prove that $G$ has an $(\mathsf{A}^{\pretr},\id)$-lift. This suffices to conclude that $F$ has an $(\mathsf{A}^{\pretr},\id)$-lift as well. 
Since $\mathsf{A}$ is lifted, we have a quasi-functor $\mathsf{g}:\mathsf{A}\to \mathsf{A}$ such that $H^*(\mathsf{g})\cong G_{\mid H^*(\mathsf{A})}^{gr}$. Up to natural isomorphism, notice that $H^0(\mathsf{g}^{\pretr})$ satisfies \eqref{eq:fgrad}. It remains to prove that $G_{\mid H^0(\mathsf{A})}^{gr} \cong H^0(\mathsf{g}^{\pretr})_{\mid H^*(\mathsf{A})}^{gr} $, which follows from direct computations. Formal standardness implies that $G\cong H^0(\mathsf{g}^{\pretr})$.
            
            Since any triangulated equivalence $F:\tr(\mathsf{A})\to \tr(\mathsf{A})$ has an $(\mathsf{A}^{\pretr},\id)$-lift, Remark \ref{rem:sfunique} and Proposition \ref{prop:strongenh} show that $\tr(\mathsf{A})$ has a strongly unique enhancement.
            \end{proof}
            \begin{rem}\label{rem:perfsffssue}
                By Proposition \ref{prop:perfectsf2} and Remark \ref{rem:performstand}, we can modify the proof of Proposition \ref{prop:sffssue} to prove that $\dc{\mathsf{A}}$ has a strongly unique enhancement under the same requirements.
            \end{rem}
            
            We now aim to prove the converse implication of Proposition \ref{prop:sffssue}. In order to do so, we need to restrict to graded categories and state a technical result (Lemma \ref{lem:cofpres}). First, we recall the following.
            \begin{fact}\emph{\cite[Theorem 1.2.10]{hovey} and \cite[Théorème 2.1, Remarque 1]{tabuada}.}\label{fact:cofibrant}
                Every DG-category is quasi-equivalent to a cofibrant\footnote{The reader not familiar with this notion may refer to \cite[Section 2]{toen}; however, we are interested only in the property highlighted by the statement.} DG-category. Moreover, let $\mathsf{C}$ be a cofibrant DG-category. Every quasi-functor $\mathsf{f}: \mathsf{C}\to \mathsf{D}$ can be represented by a DG-functor $\mathsf{f}': \mathsf{C}\to \mathsf{D}$. In particular, this means that $H^0(\mathsf{f}) \cong H^0(\mathsf{f}')$.
            \end{fact}

            \begin{lem}[Presentation via a cofibrant DG-category]\label{lem:cofpres}
                    Let $\mathsf{A}$ be a DG-category and let $\mathsf{C}$ be a cofibrant DG-category with a quasi-equivalence quasi-functor $\mathsf{f}:\mathsf{C}\to \mathsf{A}^{\pretr}$. Then we can construct a quasi-equivalence $\mathsf{h}: \mathsf{D}\to \mathsf{A}$, where $\mathsf{D}:= \mathsf{C}_{\mid H^0(\mathsf{A})}^{H^0(\mathsf{f})}$ (recall Notation \ref{nota:restrictdg}).
    \end{lem}
     \begin{proof}
        We use the notation of the statement.
                        Let $\mathsf{A}'$ be a full pretriangulated DG-subcategory of $\mathsf{A}^{\pretr}$ such that $\mathsf{A}\subset \mathsf{A}'$ and all the homotopy equivalence classes of $\mathsf{A}$ in $\mathsf{A}'$ are represented only by objects of $\mathsf{A}$. In other words, if $Y\in \mathsf{A}'$ is homotopy equivalent to an object of $\mathsf{A}$, then $Y \in \mathsf{A}$.  Notice the functor $i: H^0(\mathsf{A}')\to H^0(\mathsf{A}^{\pretr})=\tr(\mathsf{A})$, obtained from the inclusion $\mathsf{A}' \subset \mathsf{A}^{\pretr}$, is an equivalence by Lemma \ref{lem:dgtaylor}, so $(\mathsf{A}',i)$ is an enhancement of $\tr(\mathsf{A})$.
                        
                        Let us now consider the quasi-equivalence quasi-functor $\mathsf{f}':\mathsf{C}\to \mathsf{A}^{\pretr} \leftarrow \mathsf{A}'$. By Fact \ref{fact:cofibrant}, we can assume $\mathsf{f}'$ to be a DG-functor. 
                        By definition, we have that $\mathsf{f}'(\mathsf{D})\subset\mathsf{A}$ by the choice of $\mathsf{A}'$. Finally, the restriction $\mathsf{f}'_{\mid \mathsf{D}}: \mathsf{D} \to \mathsf{A}$ is the wanted quasi-equivalence $\mathsf{h}$. 
                       %
                \end{proof}

                \begin{prop}\label{prop:sueuep}
                    Let $\mathsf{B}$ be a graded category. If $\tr(\mathsf{B})$ has a strongly unique enhancement, then $\mathsf{B}$ is triangulated formal and formally standard.
                    \end{prop}
                    \begin{proof}
                        By Remark \ref{rem:sfunique}, we are reduced to check that $\mathsf{B}$ is formally standard. For this purpose, we will show that any autoequivalence $F$ on $\tr(\mathsf{B})$ satisfying \eqref{eq:fgrad} is naturally isomorphic to $H^0((F_{\mid \mathsf{B}}^{gr})^{\pretr})$. Indeed, if this is true, given any autoequivalence $G$ such that \eqref{eq:fgrad} holds, $F_{\mid \mathsf{B}}^{gr} \cong G_{\mid \mathsf{B}}^{gr}$ implies that $(F_{\mid \mathsf{B}}^{gr})^{\pretr} \cong (G_{\mid \mathsf{B}}^{gr})^{\pretr}$ by Proposition \ref{prop:dgextension}, from which $F \cong H^0((F_{\mid \mathsf{B}}^{gr})^{\pretr}) \cong H^0((G_{\mid \mathsf{B}}^{gr})^{\pretr})\cong G$, as wanted.
                        We divide our reasoning in two steps:
                        \begin{enumerate}
                            \item 
                            We choose a presentation $\mathsf{A}$ according to Lemma \ref{lem:cofpres} and a "well-behaved" associated enhancement $(\mathsf{C},H^0(\mathsf{e}))$. The meaning of its well-behaviour will become clear in the second part of the proof.
                            \item We describe a DG-functor $\mathsf{f}'$, which is a $(\mathsf{B}^{\pretr}, \id)$-lift of $F$. We conclude that $F \cong H^0(\mathsf{f}') \cong H^0((F^{gr}_{\mid \mathsf{B}})^{\pretr})$.
                        \end{enumerate}

                        We focus on item 1. Let $(\mathsf{C},E)$ be a cofibrant enhancement of $\tr(\mathsf{B})$ and define $\mathsf{A}:= \mathsf{C}_{\mid H^0(\mathsf{B})}^E$. 
                        We consider the DG-functors $\mathsf{j}:\mathsf{A}^{\pretr} \to \mathsf{C}^{\pretr}$, induced by the inclusion, and $\mathsf{h}:\mathsf{A} \to \mathsf{B}$, associated to a $(\mathsf{C},E)-(\mathsf{B}^{\pretr},\id)$-lift of the identity (which exists by Proposition \ref{prop:strongenh}) as expressed in Lemma \ref{lem:cofpres}. We define $\mathsf{e}$ to be the (quasi-equivalence) quasi-functor given by the following composition
                        \[
                         \begin{tikzcd}
                                \mathsf{C} \ar[r,hook,"\mathsf{y}"] & \mathsf{C}^{\pretr} &\mathsf{A}^{\pretr} \ar[l,"\mathsf{j}" above] \ar[r,"\mathsf{h}^{\pretr}"] & \mathsf{B}^{\pretr},
                         \end{tikzcd}  
                        \]
                        where $\mathsf{y}$ is the Yoneda embedding. 
                         We want to consider the enhancement $(\mathsf{C},H^0(\mathsf{e}))$. The "well-behaviour" discussed above is motivated by the fact that $\mathsf{A}= \mathsf{C}_{\mid H^0 (\mathsf{B})}^{H^0(\mathsf{e})}$. Let us prove it.
                        By the definition of the functors, for every $X \in H^0(\mathsf{A})$ we have 
                    \begin{equation*}
                        \begin{split}
                        H^0(\mathsf{e})(X)  & \cong  H^0(\mathsf{h}^{\pretr})H^0(   \mathsf{j})^{-1}H^0(\mathsf{y})(X) \cong H^0(\mathsf{h}^{\pretr})H^0(   \mathsf{j})^{-1}(X) \\
                        &\cong H^0(\mathsf{h}^{\pretr})(X)\cong Y \in H^0(\mathsf{B}).    
                        \end{split}
                    \end{equation*}
                    This implies that $\mathsf{A} \subset \mathsf{C}_{\mid H^0(\mathsf{B})}^{H^0(\mathsf{e})}$. Conversely, let $X \in \mathsf{C}_{\mid H^0(\mathsf{B})}^{H^0(\mathsf{e})}$. Then we have $H^0(\mathsf{e})(X) \cong Y \in H^0(\mathsf{B})$. Since $\mathsf{h}:\mathsf{A} \to \mathsf{B}$ is a quasi-equivalence, there exists $Z \in \mathsf{A}$ such that $H^0(\mathsf{h})(Z) \cong Y$. From the definitions of $\mathsf{j}$ and $\mathsf{y}$, we have $H^0(\mathsf{e}) (Z) \cong Y$. In particular, $X$ and $Z$ are homotopy equivalent. Since $\mathsf{A}$ is closed under homotopy equivalence by Lemma \ref{lem:restrictdg}, $X \in \mathsf{A}$ as wanted.

                    We prove item 2. By Proposition \ref{prop:strongenh}, $F$ has a $(\mathsf{C},H^0(\mathsf{e}))$-lift, and by Fact \ref{fact:cofibrant}, this lift can be chosen to be a DG-functor $\mathsf{f}$. Since $\mathsf{A}= \mathsf{C}_{\mid H^0(\mathsf{B})}^{H^0(\mathsf{e})}$, we have that $\mathsf{f}(\mathsf{A}) \subseteq \mathsf{A}$ because $H^0(\mathsf{e})H^0(\mathsf{f})(X)\cong FH^0(\mathsf{e})(X)\cong H^0(\mathsf{e})(X)$ for every $X \in H^0(\mathsf{A})$. We define $\mathsf{f}_{\mathsf{A}}:=\mathsf{f}_{\mid \mathsf{A}} : \mathsf{A} \to \mathsf{A}$. Notice that $\mathsf{f}^{\pretr}\mathsf{j} \cong \mathsf{jf}_{\mathsf{A}}^{\pretr}$ by Proposition \ref{prop:dgextension}. 
        
        Moreover, since $\mathsf{B}$ is graded and $\mathsf{h}:\mathsf{A}\to \mathsf{B}$ is a quasi-equivalence, we have that $H^*(\mathsf{h}): H^*(\mathsf{A}) \to H^*(\mathsf{B})=\mathsf{B}$ is a graded equivalence. Therefore, for the sake of simplicity, we can replace $\mathsf{B}$ with $H^*(\mathsf{A})$ and $\mathsf{h}$ with $H^*(\mathsf{h})^{-1} \mathsf{h}$, so that $H^*(\mathsf{h})=\id$. 
        From the definition of $H^*$, we obtain the commutative diagram of DG-categories
        \[
        \begin{tikzcd}
        \mathsf{A} \ar[r,"\mathsf{f}_{\mathsf{A}}"]\ar[d,"\mathsf{h}"] & \mathsf{A}\ar[d,"\mathsf{h}"]\\
        \mathsf{B} \ar[r,"H^*(\mathsf{f}_{\mathsf{A}})"] & \mathsf{B}
        \end{tikzcd}
        \]
        which can be extended to the pretriangulated closures by Proposition \ref{prop:dgextension}. Let $\mathsf{f}'=(H^*(\mathsf{f}_{\mathsf{A}}))^{\pretr}$. We have the following situation
        \[
        \begin{tikzcd}
        \mathsf{B}^{\pretr} \ar[d,"\mathsf{f}'"] & \mathsf{A}^{\pretr} \ar[l, "\mathsf{h}^{\pretr}" above] \ar[r,"\mathsf{j}"] \ar[d,"\mathsf{f}_{\mathsf{A}}^{\pretr}"] & \mathsf{C}^{\pretr} \ar[d,"\mathsf{f}^{\pretr}"] & \mathsf{C} \ar[lll,bend right, "\mathsf{e}" above]\ar[l,hook',"\mathsf{y}" above]\ar[r,rightsquigarrow,"H^0(\mathsf{e})"]\ar[d,"\mathsf{f}"] & \tr(\mathsf{B})\ar[d,"F"]\\
        \mathsf{B}^{\pretr} & \mathsf{A}^{\pretr} \ar[l, "\mathsf{h}^{\pretr}" above] \ar[r,"\mathsf{j}"] & \mathsf{C}^{\pretr}  & \mathsf{C} \ar[lll,bend left,"\mathsf{e}" above] \ar[l,hook',"\mathsf{y}" above]\ar[r,rightsquigarrow,"H^0(\mathsf{e})"] & \tr(\mathsf{B}).
        \end{tikzcd}
        \]
In particular, this diagram shows that $\mathsf{f}'$ is a $(\mathsf{B}^{\pretr}, \id )$-lift of $F$. We have a natural isomorphism $\eta: F \to H^0(\mathsf{f}')$, which gives a graded isomorphism $\mu: F_{\mid \mathsf{B}}^{gr} \to H^{*}(\mathsf{f}_{\mathsf{A}})$. Being $\mathsf{B}$ a DG-category with trivial differential, $\mu$ is a DG-natural isomorphism, so we can extend it to a unique DG-natural isomorphism on $\mathsf{B}^{\pretr}$ by Proposition \ref{prop:dgextension}. In particular, $(F_{\mid \mathsf{B}}^{gr})^{\pretr} \cong \mathsf{f}'$ up to DG-isomorphism, so $F \cong H^0((F_{\mid \mathsf{B}}^{gr})^{\pretr})$.
\end{proof}

            \begin{thm}\label{thm:main}
                Let $\mathsf{B}$ be a graded category. The following are equivalent:
                \begin{enumerate}
                    \item $\mathsf{B}$ is triangulated formal and formally standard;
                    \item $\tr(\mathsf{B})$ has a strongly unique enhancement;
                    \item $\dc{\mathsf{B}}$ has a strongly unique enhancement.
                \end{enumerate}
            \end{thm}
            \begin{proof}
                Proposition \ref{prop:sffssue} (remembering Example \ref{es:lifted}) and Proposition \ref{prop:sueuep} prove $1 \Leftrightarrow 2$. Remark \ref{rem:perfsffssue} deals with $1 \Rightarrow 3$, while Proposition \ref{prop:dcatra} shows $3 \Rightarrow 2$.
            \end{proof}

            \section{Free generators}\label{sec:freegen}
            Our aim is to apply Theorem \ref{thm:main}. We start with some simple examples, described by the following non-canonical definition.
            \begin{defn}
                A DG-category $\mathsf{A}$ is a \emph{free generator} if every object $X \in \dc{\mathsf{A}}$ is isomorphic to a direct summand of 
                \begin{equation}\label{eq:simplesum}
                    \bigoplus_{i \in \mathbb{Z}} \left( \bigoplus_{j=1}^{n_i} C_{i,j} [i]   \right) \qquad \text{for some } C_{i,j} \in H^0(\mathsf{A})\text{ and }n_i\ne 0\text{ for finitely many }i\text{'s.}
                \end{equation}
            \end{defn}
            \begin{es}\label{es:freegen}
                In a trivial way, any pretriangulated DG-category is a free generator, but this example is far from our idea of application, since we want to study it for graded categories (i.e. DG-categories with trivial differential). Let us give some meaningful examples.
                \begin{itemize}
                    \item The DG-category $R$ given by a semisimple ring is a free generator. Indeed, a finitely generated $R$-module is a direct summand of a free $R$-module. As a consequence, $\dc{R}$ is obtained by cones of closed morphisms
                    \[
                    \bigoplus_{i \in \mathbb{Z}} R^{n_i}[i] \to \bigoplus_{j\in \mathbb{Z}} R^{m_j} [j] ,    
                    \]
                    which are simply given by kernels and cokernels, again expressed via finitely generated $R$-modules.

                    \item Consider an algebraic finite triangulated category $\mathcal{T}$ as defined in \cite{muro} and let $\Lambda$ be the algebra of endomorphisms associated to a basic additive generator $X$. Given an enhancement ($\mathsf{C},E)$ of  $\mathcal{T}$, any object $Y \in E^{-1}(X)$, together with its endomorphisms, defines a free generator DG-category with one object (this follows from the fact that $\mathcal{T}$ is equivalent to the category of finitely generated projective (right) $\Lambda$-modules).
                \end{itemize}
            \end{es}
            Roughly speaking, the following lemma tells us that whenever a ring is a free generator, then it is a free generator also for its associated periodic triangulated categories (the notion of periodic triangulated category is given, for instance, in the introduction of \cite{saito23}).
     \begin{lem}\label{lem:freegenperiod}
        Let $R$ be a ring, i.e. a graded ring concentrated in degree 0, and let us consider $R[t,t^{-1}]$ with degree induced by setting $\deg(t)=n>0$. If $R$ is a free generator, then so is $R[t,t^{-1}]$.
    \end{lem}
    \begin{proof} 
    Let us consider the inclusion $R \to R[t,t^{-1}]$, which is a (differential) graded morphism. Then we can extend it to a DG-functor $\pi: \perf(R) \to \perf(R[t,t^{-1}])$. 
    By \cite[Proposition 3.36]{saito23}, a DG-module of $R[t,t^{-1}]$ can be considered as an $n$-periodic DG-module of $R$, in the sense of \cite[Definition 3.1]{saito23}.
    Accordingly, $\pi$ is explicitly expressed by (cf. \cite[Definition 3.8]{saito23})
    \begin{equation*}
    \begin{split}
    \pi(M)^k&:= \bigoplus_{\bar{\ell}= k} M^\ell , \qquad d_{\pi(M)}^k := \bigoplus_{\bar{\ell}=k }d_M^\ell
    \end{split}
    \end{equation*}
    where $\bar{\ell}$ indicates the representative of $\ell$ in $\mathbb{Z}/n\mathbb{Z}$ (the behaviour of $\pi$ on morphisms is expressed accordingly).
    From this description, it is easy to notice that $\pi$ is essentially surjective. Since $\pi$ is also additive and preserves suspensions, the statement follows.
    \end{proof}
            \begin{prop}\label{prop:simplesue}
                Let $\mathsf{A}$ be a free generator DG-category. Then $\mathsf{A}$ is formally standard. 
            \end{prop}
            \begin{proof}
                Define $\mathsf{A}_{\operatorname{add}}$ the full DG-subcategory of $\mathsf{A}^{\pretr}$ whose objects are of the form \eqref{eq:simplesum}. Notice that $H^0(\mathsf{A}_{\operatorname{add}}) \subset \tr (\mathsf{A})$. Let $F: \tr(\mathsf{A})\to \tr(\mathsf{A})$ a triangulated equivalence such that its graded restriction is naturally isomorphic to the identity. Therefore, $F_{\mid H^0(\mathsf{A}_{\operatorname{add}})}$ is naturally isomorphic to the identity. Let $G$ be the composition
                \[
                \begin{tikzcd}[column sep=large]
                    H^0(\mathsf{A}_{\operatorname{add}}) \ar[r,"F_{\mid H^0(\mathsf{A}_{\operatorname{add}})}"] &\tr(\mathsf{A}) \ar[r,hook] &\dc{\mathsf{A}}
                \end{tikzcd}    
                \]
            Since $\mathsf{A}$ is a free generator, $\dc{\mathsf{A}}$ is the idempotent completion of $H^0(\mathsf{A}_{\operatorname{add}})$; \cite[Proposition 1.3]{balmerschlichting} gives a unique extension $H:\dc{\mathsf{A}} \to \dc {\mathsf{A}}$ of $G$, which is therefore an extension of $F$. Further, from the same proposition the natural isomorphism $F_{\mid H^0(\mathsf{A}_{\operatorname{add}})} \to \id$ extends to a natural isomorphism $H \to \id$. By restricting this natural isomorphism, $F$ is naturally isomorphic to the identity. Remark \ref{rem:gradrestid} concludes the proof.
            \end{proof}
            \begin{cor}\label{cor:stronguniquefree}
                Let $\field$ be a field. Given a free generator $\field$-algebra $\Lambda$ with finite projective dimension $d$ as a $\Lambda$-bimodule, then $\dc{\Lambda[t,t^{-1}]}$ has a strongly unique $\field$-linear enhancement for any $t$ homogeneous of degree greater or equal than $d$.
            \end{cor}    
                \begin{proof}
                Under these assumptions, \cite[Corollary 4.8]{saito23} shows that $\Lambda[t,t^{-1}]$ is intrinsically formal. By Example \ref{es:intrformalstrong}, Lemma \ref{lem:freegenperiod},  Proposition \ref{prop:simplesue} and Theorem \ref{thm:main}, we conclude.
            \end{proof}

            \begin{es} Here we list some examples following from the previous corollary. \label{es:Kmod}\item 
                \begin{itemize}
                    \item The triangulated category $\field\fmod$ with suspension the identity is $\tr(\field[t,t^{-1}])$ with $t$ of degree $1$, so it has a strongly unique $\field$-linear enhancement by Corollary \ref{cor:stronguniquefree}. Notice this is not the case when $\field=\mathbb{F}_p$ (with $p$ a prime) and we consider linearity over $\ring=\mathbb{Z}$ (see \cite{schlichting} and \cite[Corollary 3.10]{canonacostellari}). In particular, $\mathbb{F}_p [t,t^{-1}]$ is not intrinsically formal as a $\mathbb{Z}$-linear graded ring.
                \item  Let $\field$ be a perfect field and let $\Lambda$ be a semisimple finite-dimensional $\field$-algebra. In this case, $\Lambda$ is a projective $\Lambda$-bimodule by \cite[Corollary b, p. 192]{pierce}. Then, by Example \ref{es:freegen} and Corollary \ref{cor:stronguniquefree}, $\dc{\Lambda[t,t^{-1}]}$ has a strongly unique $\field$-linear enhancement for any homogeneous $t$ of positive degree.
                \end{itemize}
            \end{es}
            \begin{rem}
                Let us briefly discuss the example of a non-unique $\field$-linear enhancement provided by Rizzardo and Van den Bergh in \cite{rizzvdbnonunique}.
                
                Let $\field$ be a field and $\mathbb{F}:= \field (x_1 , \dots , x_n)$ with $n>0$ even. Then $\dc{\mathbb{F}[t,t^{-1}]}$ with $\deg(t)=n$ has a non-unique enhancement, as shown in \cite{rizzvdbnonunique}.
                We notice that this example carefully avoids any situation described above. As discussed in Example \ref{es:freegen}, $\mathbb{F}[t,t^{-1}]$ is a free generator because $\mathbb{F}$ is semisimple.
                However, it is not finite-dimensional, because, as seen in Example \ref{es:Kmod}, this will not work for a perfect field. 
                
                As one may expect from the viewpoint depicted in this article, the proof in \cite{rizzvdbnonunique} shows explicitly that $\mathbb{F}[t,t^{-1}]$ is not intrinsically formal by deforming the graded algebra into a different minimal $A_\infty$-algebra.
            \end{rem}
            
            \section{D-standardness and K-standardness}\label{sec:kdstandard}
            In this section we consider the notions of D-standard and K-standard categories introduced in \cite{chenye}, and show that for a very large class of examples they are, in fact, equivalent to strongly unique enhancement. 
            We emphasize that these results hold for $\ring$-linearity, where $\ring$ is any commutative ring.
            
            \begin{defn}\label{def:kdstandard}
                Let $\mathcal{E}$ be an exact category and consider $F: \derb(\mathcal{E}) \to \derb(\mathcal{E})$ a triangulated equivalence.
                Then $\mathcal{E}$ is \emph{D-standard} if the following implication holds: 
                \begin{itemize}
                    \item[($\spadesuit$)] Whenever $F(\mathcal{E}) \subseteq \mathcal{E}$ and $\eta_0: F_{\mid \mathcal{E}} \to \id_{\mathcal{E}}$ is a natural isomorphism, there exists a natural isomorphism $ \eta: F \to \id$ extending $\eta_0$.
                \end{itemize}
                An additive category $\mathcal{A}$ is \emph{K-standard} if it is D-standard as an exact category.
                \end{defn}

            \begin{lem}\label{lem:kstandard}
                Let $F:\derb(\mathcal{E})\to \derb(\mathcal{E})$, with $\mathcal{E}$ an exact category, be a triangulated equivalence. Then $(\spadesuit)$ holds if and only if the following is satisfied.
                \begin{itemize}
                    \item[$(\clubsuit)$] Whenever  $F(\mathcal{E})\subset \mathcal{E}$ and $F_{\mid \mathcal{E}}\cong \id_{\mathcal{E}}$, then $F \cong \id$.
                \end{itemize}
            \end{lem}
            \begin{proof}
                This result is analogous to \cite[Lemma 3.5]{chenye}. The fact that $(\spadesuit)$ implies $(\clubsuit)$ is obvious. Conversely, let $\eta_0 : F_{\mid \mathcal{E}} \to \id_{\mathcal{E}}$. By $(\clubsuit)$, there exists $\mu: F \to \id$. In particular, $\mu_{\mid \mathcal{E}}:F_{\mid \mathcal{E}}\to \id_{\mathcal{E}}$, so we can consider $\derb(\eta_0\mu_{\mid \mathcal{E}}^{-1}) \mu : F \to \id$: by definition, such natural isomorphism restricted to $\mathcal{E}$ is $\eta_0$. This concludes the proof.
            \end{proof}

            \begin{prop}\label{prop:kfstandard}
                An additive category $\mathcal{A}$ is K-standard if and only if $\kcomb(\mathcal{A})$ has a strongly unique enhancement.
            \end{prop}
            \begin{proof}
                We recall that $\mathcal{A}$ is triangulated formal by Proposition \ref{prop:asf}. Then Remark \ref{rem:gradrestid} and Lemma \ref{lem:kstandard} shows that $\mathcal{A}$ is K-standard if and only if it is formally standard. Theorem \ref{thm:main} concludes the proof.
            \end{proof}
            \begin{es}
                A Krull-Schmidt additive category $\mathcal{A}$ is an \emph{Orlov category} if 
                \begin{enumerate}
                    \item The endomorphism ring of each indecomposable is a division ring;
                    \item There is a degree function $\deg : \operatorname{ind}(\mathcal{A}) \to \mathbb{Z}$, where $\operatorname{ind}(\mathcal{A})$ is the set of all indecomposables, such that $\deg X \le \deg Y$ implies $\hom(X,Y)=0$ whenever $X \not \cong Y$.
                \end{enumerate}
                As proved in \cite[Proposition 4.6]{chenye}, 
                an Orlov category $\mathcal{A}$ is K-standard, so $\kcomb(\mathcal{A})$ has a strongly unique enhancement.
            \end{es}

            \begin{lem}\label{lem:dstandard}
                Let $\mathcal{E}$ be an exact category, and let $\mathcal{E}_{\dg}:= \derb_{\dg}(\mathcal{E})_{\mid \mathcal{E}}$ be the DG-category as in Example \ref{es:lifted}. Then $\mathcal{E}$ is D-standard if and only if $\mathcal{E}_{\dg}$ is formally standard.
            \end{lem}
            \begin{proof}
                By Lemma \ref{lem:kstandard} and Proposition \ref{prop:natfgrad}, $\mathcal{E}$ is D-standard if and only if any triangulated equivalence $F$ such that $F(X)= X$ for $X \in \mathcal{E}$ and $F_{\mid \mathcal{E}}\cong {\id}_{\mathcal{E}}$ is naturally isomorphic to the identity. By Remark \ref{rem:gradrestid}, we are reduced to prove that $F_{\mid \mathcal{E}}\cong \id_{\mathcal{E}}$ if and only if $F_{\mid H^*(\mathcal{E}_{\dg})}^{gr} \cong \id_{H^*(\mathcal{E}_{\dg})}$. This follows from Proposition \ref{prop:0detgr}. 
            \end{proof}

            We now aim to show the analogous of Proposition \ref{prop:kfstandard} for derived categories.

            \begin{prop}\label{prop:dstandtrunc}
                Let $\mathcal{E}$ be an exact category, and consider $F$ a triangulated autoequivalence of $\derb(\mathcal{E})$ such that $F(\mathcal{E}) \subset \mathcal{E}$ and $F_{\mid \mathcal{E}} \cong \id_{\mathcal{E}}$. Let $(\mathsf{C},E)$ be any enhancement of $\derb(\mathcal{E})$. If $F$ has a $(\mathsf{C},E)$-lift, then $F$ is naturally isomorphic to the identity.
            
                Moreover, the identity of $\mathsf{C}$ is the only quasi-functor lifting the identity of $\derb(\mathcal{E})$. Consequently, any autoequivalence of $\derb(\mathcal{E})$ has at most one $(\mathsf{C},E)$-lift.
            \end{prop}
            
             \begin{proof}
                Without loss of generality, assume $\mathsf{C}$ is a cofibrant DG-category. By Fact \ref{fact:cofibrant}, there exists a DG-functor $\mathsf{f}: \mathsf{C} \to \mathsf{C}$ which is a $(\mathsf{C},E)$-lift of $F$.
                We now consider $\mathsf{f}_{\mid \mathcal{E}} : \mathsf{C}_{\mid \mathcal{E}} \to \mathsf{C}_{\mid \mathcal{E}}$. Defined the quasi-equivalence quasi-functor 
                \[
                             \begin{tikzcd}(\mathsf{C}_{\mid \mathcal{E}})^{\pretr} \ar[r,"\mathsf{j}"] & \mathsf{C}^{\pretr} & 
                             \mathsf{C},\ar[l, hook', "\mathsf{y}" above]
                             \end{tikzcd}    
                             \]
                             where $\mathsf{j}$ is induced by inclusion, notice that $\mathsf{y} \mathsf{f}\cong \mathsf{f}^{\pretr} \mathsf{y}$ and  $\mathsf{j}(\mathsf{f}_{\mid \mathcal{E}})^{\pretr}\cong \mathsf{f}^{\pretr} \mathsf{j}$ by Proposition \ref{prop:dgextension}. 
                             
                             Moreover, we are able to construct the following commutative diagram
                             \[
                             \begin{tikzcd}
                                H^0(\mathsf{C}_{\mid \mathcal{E}}) \ar[d,"\id"] & \tau_{\le 0}\mathsf{C}_{\mid \mathcal{E}} \ar[l,"\simeq"] \ar[r,"\mathsf{p}_{\le 0}"]\ar[d,"\tau_{\le 0}\mathsf{f}_{\mid \mathcal{E}}"] & \mathsf{C}_{\mid \mathcal{E}}\ar[d,"\mathsf{f}_{\mid \mathcal{E}}"]\\
                                H^0(\mathsf{C}_{\mid \mathcal{E}}) & \tau_{\le 0}\mathsf{C}_{\mid \mathcal{E}} \ar[l,"\simeq"] \ar[r,"\mathsf{p}_{\le 0}"] & \mathsf{C}_{\mid \mathcal{E}}
                             \end{tikzcd}   
                             \]
                             by assumption (indeed $H^0(\mathsf{f}_{\mid \mathcal{E}})= F_{\mid \mathcal{E}}\cong \id$). The commutative diagram obtained by taking the pretriangulated closures shows that $(\mathsf{f}_{\mid \mathcal{E}})^{\pretr}$ is the identity quasi-functor by the universal property of the DG-quotient (see item 1 of Definition/Proposition \ref{defp:dgquotient} and Corollary \ref{cor:dbunique}). This suffices to show that $\mathsf{f}$ is the identity as well because $\mathsf{f} \mathsf{y}^{-1} \mathsf{j}\cong \mathsf{y}^{-1} \mathsf{j}(\mathsf{f}_{\mid \mathcal{E}})^{\pretr}\cong \mathsf{y}^{-1} \mathsf{j}$, where $\mathsf{j}$ and $\mathsf{y}$ are quasi-equivalences. In particular, $F\cong E H^0(\mathsf{f})E^{-1} \cong EE^{-1} \cong \id$, as wanted.
             \end{proof}

\begin{thm}\label{thm:exdstand} Let $\mathcal{E}$ be an exact category. Then $\mathcal{E}$ is D-standard if and only if $\derb(\mathcal{E})$ has a strongly unique enhancement.
\end{thm}
\begin{proof}
    Recall that $\mathcal{E}_{\dg}$ is triangulated formal by Corollary \ref{cor:dbunique}. If $\mathcal{E}$ is D-standard, by Lemma \ref{lem:dstandard} $\mathcal{E}_{\dg}$ is also formally standard. Proposition \ref{prop:sffssue} shows that $\derb(\mathcal{E})$ has a strongly unique enhancement, since $\mathcal{E}_{\dg}$ is always lifted (see Example \ref{es:lifted}). The converse implication follows from Proposition \ref{prop:dstandtrunc} and Proposition \ref{prop:strongenh}.
\end{proof}

                
            
            \begin{prop}\label{prop:stronguniquedstand}
                Let $\mathcal{A}$ be an abelian category with enough projective objects. We denote with $\operatorname{Proj}(\mathcal{A})$ its subcategory of projective objects. If $\kcomb(\operatorname{Proj}(\mathcal{A}))$ has a strongly unique enhancement, then $\derb(\mathcal{A})$ has a strongly unique enhancement.
            \end{prop}
            \begin{proof}
            It immediately follows from Proposition \ref{prop:kfstandard}, \cite[Theorem 6.1]{chenye} and Theorem \ref{thm:exdstand}.
            \end{proof}


            We now state some examples obtained by applying Theorem \ref{thm:exdstand}.
            \begin{cor}\label{cor:hereditary}
                Let $\mathcal{A}$ be a hereditary abelian category (i.e. $\ext^i=0$ for $i>1$). Then its bounded derived category $\derb(\mathcal{A})$ has a strongly unique enhancement.
                (cf. \cite[Corollary 5.6]{chenye}).
            \end{cor}
            \begin{es}
                Let $R$ be any ring (recall Convention \ref{conv:field}). If $R$ is (right) hereditary, the bounded derived category of (all right) $R$-modules has a strongly unique enhancement. If $R$ is (right) semihereditary and Noetherian, the bounded derived category of finitely generated $R$-modules has a strongly unique enhancement. Dually, the result holds for left modules.
            \end{es}

            To prove that bounded derived categories of smooth projective varieties have a strongly unique enhancement, Lunts and Orlov in fact showed D-standardness using the notion of ample sequence \cite{luntsorlov}. Here we consider a generalization due to Canonaco and Stellari \cite[Definition 2.9]{suppcanste}.
            \begin{defn}
                Given an abelian category $\mathcal{A}$ and a set $I$, we say that $\lbrace P_i \rbrace_{i \in I}\subset \mathcal{A}$ is an \emph{almost ample set} if, for any $A\in \mathcal{A}$, there exists $i \in I$ such that
                \begin{enumerate}
                    \item There is a natural number $k$ such that $P_i^{\oplus k} \to A$ is an epimorphism;
                    \item $\hom(A,P_i)=0$.
                \end{enumerate}
            \end{defn}

            \begin{es} Given an algebraic space $X$ proper over an Artinian ring with depth $\ge 1$
                at every closed point, the category of coherent sheaves $\operatorname{Coh}(X)$ has an almost ample set (see \cite[Lemma 3.3.2]{olanderphd}). Another class of examples is given by \cite[Proposition 2.12]{suppcanste}.
                \end{es}

            Finally, we have the following result for derived categories with a geometric flavour.
            \begin{prop} \label{prop:almostample}
                Let $\mathcal{A}$ be an abelian category with an almost ample set. Then $\derb(\mathcal{A})$ has a strongly unique enhancement.
            \end{prop}
            \begin{proof}[Sketch of proof]
                We want to show that $\mathcal{A}$ is D-standard and apply Theorem \ref{thm:exdstand}. The proof is almost identical to the one of \cite[Proposition 3.7]{suppcanste}, which is a generalization of \cite[Proposition 3.4.6]{orlov2}. For a self-contained proof, the reader may refer to \cite[Theorem 5.58]{lorenzinphd}.
            \end{proof}
            \appendix
            \section{On bounded derived categories of exact categories}\label{app:derbex}
            Here we present two results that should be expected by experts, but that we did not find in the literature. 
              


            The following generalizes the last part of the statement of \cite[Lemma 3.1]{krauseaus}. 
            \begin{lem}\label{lem:genconfl}
            Let $\mathcal{E}$ be an exact category. Then $\operatorname{Ac}^{\operatorname{b}}(\mathcal{E})\subset \kcomb(\mathcal{E})$, the category of acyclic complexes, is the triangulated envelope of the full subcategory given by conflations (intended as complexes).
            \end{lem}
            \begin{proof}
                Let $X=(X^i,d^i)$ be a bounded acyclic complex. Up to shift, we can assume $X^i=0$ for $i<1$ and $i>n$ for some $n>3$. Let us write
         \[\arraycolsep=0.5pt
         \begin{array}{rlrll}
             X:= &\dots \to 0 \to  X^1 \to X^2 \to \dots&&&\to X^n \to 0 \to \dots\\
             X^{\le n-2} :=& \dots \to 0 \to X^1 \to X^2 \to \dots& \to X^{n-2} &\overset{p}{\to} \coker (d^{n-3}) & \to 0 \to \dots \\
             X^{\ge n-1}[-1]:= && \dots \to 0 &\to \coker (d^{n-3}) &\overset{j}{\to} X^{n-1} \overset{-d^{n-1}}{\to} X^{n} \to 0\to \dots
         \end{array}
         \]
         where the composition of $p: X^{n-2} \to \coker (d^{n-3})$ and $j: \coker (d^{n-3}) \to X^{n-1}$ gives $-d^{n-2}$.
     
         From direct computations, we have that $X \cong \cone(f)$, where $f: X^{\ge n-1}[-1]\to X^{\le n-2}$ is defined to be the identity on $\coker(d^{n-3})$ and 0 elsewhere. The conclusion now follows from an induction.
            \end{proof}

We recall that the Ext-groups of an exact category $\mathcal{E}$ are given exactly as in the case of abelian categories. The $n$-extensions of $A$ and $B$ are, up to the same equivalence relation, acyclic complexes of length $n+2$ of the form: 
            \[
            \begin{tikzcd}
             0 \ar[r]& B \ar[r,"\xi_0"]& X_1 \ar[r,"\xi_1"] & X_2 \ar[r,"\xi_2"]&\dots \ar[r,"\xi_{n-1}"]& X_n \ar[r,"\xi_n"]& A \ar[r]& 0
             \end{tikzcd}
                    \] 
                For a more precise discussion, the reader may consult \cite[Definition A.4]{lorenzin1}.

\begin{prop}\label{prop:0detgr} Let $\mathcal{E}$ be an exact category, and consider $F:\derb(\mathcal{E})\to \derb(\mathcal{E})$ a triangulated equivalence such that $F(X)=X$ for all $X \in \mathcal{E}$. Then $F_{\mid \mathcal{E}}$ determines uniquely $F_{\mid H^*(\mathcal{E}_{\dg})}^{gr}$.
\end{prop}

 \begin{proof}
    By \cite[Corollary A.7.1 and Proposition A.7]{positselski}, $H^*(\mathcal{E}_{\dg})$ is simply the category of the Ext-groups, i.e. \[\hom_{H^*(\mathcal{E}_{\dg})} (X,Y)= \bigoplus_i \ext^i (X,Y)[-i]\] 
    for every $X,Y \in \mathcal{E}$.
    Since every morphism in $\hom(X,Y[1])$ is associated to a conflation (as expressed in \cite{dyer}), given another triangulated autoequivalence $(G,\mu)$ such that $G_{\mid \mathcal{E}}=F_{\mid \mathcal{E}}$ (so $G(X)=X$ for all $X \in \mathcal{E}$ as well), we have 
    the following isomorphism of distinguished triangles
    \[
    \begin{tikzcd}
        Y \ar[r,"Ff"]\ar[d,"\id"] & Z \ar[r,"Fg"]\ar[d,"\id"] & X\ar[r,"\eta_{Y}Fh"]\ar[d,dashed] & Y[1]\ar[d,"\id"]\\
        Y \ar[r,"Gf"] & Z \ar[r,"Gg"] & X\arrow{r}{\mu_{Y} Gh} & Y[1],
    \end{tikzcd}
    \]
    where the first two vertical arrows are the identity because $G_{\mid \mathcal{E}}=F_{\mid \mathcal{E}}$. From the universal property of the cokernel $X$ in $\mathcal{E}$, the dashed morphism is also the identity, so that $\eta_{Y}Fh = \mu_{Y}Gh$. As every extension is obtained by Yoneda products of elements in the first Ext-group (cf. the beginning of the proof of \cite[Proposition A.7]{lorenzin1}), a simple induction concludes the proof.
 \end{proof}

\bibliographystyle{alphaabbr}
{\footnotesize 
\bibliography{biblio}}
\end{document}